\documentclass[11pt]{amsart}
\usepackage[margin=0.7in]{geometry}% Change the margins here if you wish.
 \usepackage{pst-node}
\usepackage{multirow}
\usepackage{float}
\usepackage{stmaryrd}
\usepackage{tikz}
\usepackage{quiver}
\usepackage{thmtools}
\usepackage{tikz}
\usepackage{xcolor}
\usetikzlibrary{arrows.meta,shapes.geometric}
\usetikzlibrary{shapes.geometric,positioning}



\usepackage{amsmath,amsthm,amssymb}
\usepackage{nicematrix}
\usepackage{graphicx}
\usepackage{euscript}
\usepackage{color}
\usepackage{enumerate}
\usepackage{multicol}
\usepackage{thmtools}
\usepackage{amsfonts}
\usepackage{bm}
\usepackage[colorinlistoftodos]{todonotes}
\usepackage{hyperref}
\usepackage{xcolor}
\hypersetup{
    colorlinks,
    linkcolor={blue!80!pink},
    citecolor={blue!80!pink},
    urlcolor={blue!100!black}
}
\usepackage[capitalise]{cleveref}
\usepackage{mathtools}
\usepackage[colorinlistoftodos]{todonotes}
\usepackage{verbatim}
\usepackage{pict2e}
\usepackage[T1]{fontenc}
\newtheorem{theorem}{Theorem}[section]
\newtheorem*{claim}{Claim}
\newtheorem{prop}[theorem]{Proposition}
\newtheorem{corollary}[theorem]{Corollary}
\newtheorem{defn}[theorem]{Definition}

\newtheorem{remark}[theorem]{Remark}
\newtheorem{lemma}[theorem]{Lemma}

\usepackage[english]{babel}
\newcommand{\doublearrow}[1]{\xrightarrow[]{#1}\mathrel{\mkern-14mu}\rightarrow}

\renewcommand{\url}[1]{\href{#1}{Available online}}
\makeatletter
\newcommand{\adjunction}[4]{%
  #1\colon #2%
  \mathrel{\vcenter{%
    \offinterlineskip\m@th
    \ialign{%
      \hfil$##$\hfil\cr
      \longrightharpoonup\cr
      \noalign{\kern-.3ex}
      \smallbot\cr
      \longleftharpoondown\cr
    }%
  }}%
  #3 \noloc #4%
}
\newcommand{\longrightharpoonup}{\relbar\joinrel\rightharpoonup}
\newcommand{\longleftharpoondown}{\leftharpoondown\joinrel\relbar}
\newcommand\noloc{%
  \nobreak
  \mspace{6mu plus 1mu}
  {:}
  \nonscript\mkern-\thinmuskip
  \mathpunct{}
  \mspace{2mu}
}
\newcommand{\smallbot}{%
  \begingroup\setlength\unitlength{.15em}%
  \begin{picture}(1,1)
  \roundcap
  \polyline(0,0)(1,0)
  \polyline(0.5,0)(0.5,1)
  \end{picture}%
  \endgroup
}
\makeatother

\usepackage{braket}
\usepackage{mathrsfs}

\newcommand{\colim} {\operatornamewithlimits{{colim}}}
\newcommand{\Sp}{\operatorname{Sp}}

\newcommand{\id}{\operatorname{id}}
\newcommand{\Ker}{\operatorname{Ker}}

\newcommand{\fib}{\operatorname{fib}}
\newcommand{\sphere}{\mathbb{S}}
\newcommand{\Pic}{\operatorname{Pic}}
\newcommand{\KU}{\operatorname{KU}}
\newcommand{\Map}{\operatorname{Map}}
\newcommand{\K}{\operatorname{K}}

\newcommand{\ROS}{\operatorname{RO}^{\sigma}}
\newcommand{\KS}{\mathbb{S}_{\K(1)}}

\newcommand{\ba}{\bar{\alpha}}
\newcommand{\res}{\operatorname{Res}_e^{C_p}}
\newcommand{\tr}{\operatorname{Tr}_e^{C_p}}

\newtheorem{maintheorem}{Theorem}

\title{The equivariant $\K(1)$-local sphere at odd primes}
\author{Pengkun Huang}
\date{\today}

\begin{document}

\maketitle
\begin{abstract}
    We compute the $\operatorname{RO}^{\sigma}(C_p)$-graded Mackey functors $\underline{\pi}_* L_{\KU_{C_p}/p}\sphere_{C_p}$.
\end{abstract}
\tableofcontents
\section{Introduction}
The chromatic viewpoint organizes stable homotopy theory by separating phenomena according to their height. At height one, the nonequivariant local sphere $L_{K(1)}S$ is well understood, and its homotopy groups were computed by Bousfield and Ravenel \cite{Ravenel1984Localization, Bousfield1979Localization}. Genuine equivariant stable homotopy theory introduces a second source of complexity: chromatic behavior may vary with the isotropy subgroup, and equivariant homotopy groups carry Euler-class, restriction, transfer, and representation-graded structure. Thus, even at height one, the equivariant local sphere contains information that has no nonequivariant analogue. A systematic account of this subgroup-dependent chromatic picture can be found in \cite{BehrensCarlisle2025}. 

Fix an odd prime $p$ and let $G=C_p$. Let $\KU_G$ denote genuine $G$-equivariant complex $K$-theory. We study the $C_p$-equivariant spectrum $L_{\KU_G/p}\sphere_{C_p}$. By \cite[A.4.12]{Balderrama2024TotalPowerOperations}, there is an equivalence
\begin{equation*}
    L_{\KU_{C_p}/p} \sphere_{C_p} \cong F({EC_p}_+, i_* \sphere_{\K(1)}),
\end{equation*}
where $i_*$ is the inflation functor, which equips a spectrum with the trivial $G$-action. Thus, its geometric fixed points are
\begin{equation*}
    (L_{\KU_{C_p}/p} \sphere_{C_p})^{\Phi e} \cong \sphere_{\K(1)}, \quad (L_{\KU_{C_p}/p} \sphere_{C_p})^{\Phi C_p} \cong \sphere_{\K(1)}^{tC_p},
\end{equation*}
where $\sphere_{\K(1)}^{tC_p}$ is a rational spectrum with homotopy groups concentrated in degree $0$ and $-1$, isomorphic to the cyclotomic field $Q_p$ \cite[Example I.2.3(iii)]{NikolausScholze2018}. In the notation of equivariant chromatic homotopy theory, one may equivalently identify this spectrum with $L_{K(e,1)} \sphere_{C_p}$. 

For $G=C_2$, the \(RO(C_2)\)-grading is sufficient for many calculations of genuine spectra. For example, Guillou–Isaksen \cite{guillou2024c_2} computed the $C_2$-equivariant stable stems in stems $0 \leq s \leq 25$ and $-1 \leq t \leq 7$. Combining the computation of $\mathbb{R}$-motivic stems and the comparison theorem by \cite{BelmontGuillouIsaksen2021Comparison} and \cite{BelmontIsaksen2022RealMotivicStems}, certain stems are calculated in coweight $8 \leq t \leq 11$.  At an odd prime, however, the \(RO(C_p)\)-graded homotopy groups alone do not detect equivalences. Hahn, Senger, and Wilson \cite{HSW23} introduced the spoke sphere \(S^\sigma\), defined as the cofiber of $ C_p{}_+ \longrightarrow S_{C_p}$, and the associated spoke graded homotopy groups, denoted by $\pi_{i,j}^{C_p} X$. The spoke-graded homotopy groups detect equivalences of genuine \(C_p\)-spectra, and therefore retain strictly more information than the single \(RO(C_p)\)-graded groups. Recent computations of the \(C_3\)-equivariant stable stems by Hou and Zhang illustrate the effectiveness of this grading in concrete calculations \cite{hou2025c3equivariantstablestems}. The present paper may be viewed as an all-degree calculation of the height-one contribution to the spoke-graded \(C_p\)-equivariant stable stems.

There have been several related calculations of equivariant $\KU_G$-local spheres. At the prime \(2\), Balderrama \cite{Balderrama2026C2K1LocalSphere} computed the \(RO(C_2)\)-graded Green functor of $ L_{\KU_{C_2}/2}S_{C_2}$. At odd primes, Carawan et al. \cite{carawan2023homotopy} computed the $\mathbb{Z}$-graded homotopy Mackey functor of the \(\KU_G\)-local sphere for odd-primary groups. More recently, Li \cite{Li2026KUGLocalSphere} computed the \(\mathbb Z\)-graded homotopy Mackey functors of the \(\KU_G\)-local sphere for finite abelian groups. In this paper, we study the finer spoke grading and determine the additive groups together with their Euler-class, restriction, and transfer structure.

\subsection*{Outline of the Calculation and Main Results} The starting point of the calculation is a translation from genuine equivariant homotopy groups to generalized cohomology groups. Let 
\begin{equation*}
    P_j = \Sigma^{\infty}(BC_p)_j^{\infty}
\end{equation*}
denote the $j$-th stunted lens spectrum, implicitly localized at $p$. For every spectrum Y, \cref{spoketocohomo} identifies the spoke graded homotopy groups (defined in (\ref{defofspoke})) as:
\begin{equation*}
    \pi_{i,j}^{C_p} F({EC_p}_+, i_*Y) \cong Y^{j-i}(P_j).
\end{equation*}
Consequently, the spoke-graded calculation for $L_{\KU_G/p} \sphere_{C_p}$ reduces to computing the $\sphere_{\K(1)}$-cohomology of the spectra $P_j$. When $j$ is odd, the calculation of cohomology groups come from its $\K(1)$-local decomposition. Let $\omega\in \mathbb{Z}_p^{\times}$ be a primitive $(p-1)$-st root of unity, and let $\sphere_{\K(1)}[\omega^r]$ denote the corresponding invertible $\K(1)$-local spectrum. We prove that (see \cref{decomposeatpositive} and \ref{decomposeatnegative}) for every $n\in \mathbb{Z}$,
\begin{equation*}
    L_{\K(1)}P_{2n+1} \cong \bigvee_{r=1}^{p-1} \sphere_{\K(1)}[\omega^r].
\end{equation*}
This decomposition is obtained from the \(p\)-local stable splitting of \(BC_p\), together with the action of Adams operations on \(p\)-complete complex \(K\)-theory. It immediately determines the cohomology of all odd tails and hence the odd spoke weights. The even tails are subtler. They fit into cofiber sequences \[ S^{2n} \longrightarrow P_{2n} \longrightarrow P_{2n+1}, \] so their cohomology is governed by extension problems between the cohomology of the odd tail and that of the bottom cell. These extensions are generally nontrivial. We determine them by studying the cellular attaching maps of stunted lens spectra and by computing the relevant Toda brackets in the homotopy groups of the \(K(1)\)-local sphere. This gives the complete additive answer in even spoke weights.
\begin{maintheorem}\label{thm:intro-additive}
   Let \(p\) be an odd prime. The spoke-graded homotopy groups \[ \pi_{i,j}^{C_p}L_{\KU_G/p}S_{C_p} \] are determined for every \(i,j\in\mathbb Z\). For odd spoke weight \(j=2k+1\), there is an isomorphism
   \begin{equation*}
       \pi_{i,2k+1} L_{\KU_G/p} \sphere_{C_p} \cong
          \left\{
        \begin{aligned}
           & \mathbb{Z}/p^{v_p(k-m)+1},  & \text{if } i = 2m , \text{ with } m \neq k;\\
           & \mathbb{Z}_p, & \text{ if } i= 2k \text{ or } 2k+1.\\
           & 0, & \text{otherwise.}
        \end{aligned}
        \right.
    \end{equation*}
    For even spoke weights, the answer is given explicitly in terms of \(v_p(k)\), \(v_p(m)\), and \(v_p(k-m)\), including all additive extension problems. The complete formulas are recorded in \cref{Main result}.
\end{maintheorem}
The even-weight answer is more complicated than the odd-weight answer. In particular, when \(j=2k\), the groups may contain both free \(\mathbb Z_p\)-summands and nontrivial extensions between finite cyclic \(p\)-groups. Determining these extensions is an essential part of the calculation rather than a consequence of the associated graded groups. Our second main result determines the genuine equivariant structure on these groups.

\begin{maintheorem}
    \label{thm:intro-structure} 
    The equivariant structure maps $a_\sigma, \res, \tr $ on \(\pi_{\star,\star}^{C_p}b(S_{K(1)})\) are determined in every bidegree. These maps are described explicitly, up to the indicated choices of generators, in Section \ref{equivariantstructure}. 
    
    In particular, all restriction maps and transfer maps in odd spoke weights are zero except in bidegrees
      $ (rq-1,rq-1), r\neq 0$ 
      and
    $ (0,2k+1), k\neq 0$
    respectively. And in both cases, their image is a cyclic subgroup of order $p$.
    \end{maintheorem} 
 Together, Theorems~\ref{thm:intro-additive} and \ref{thm:intro-structure} determine the spoke-graded homotopy Mackey functor structure of \(L_{\KU_{C_p}/p}S_{C_p}\).
 
Finally, we reinterpret the cohomological Atiyah--Hirzebruch spectral sequence as a \(\lambda\)-torsion Bockstein spectral sequence associated to the skeletal tower of \(P_{2n+1}\). In this formulation, a class supports a differential of length \(r\) precisely when its image in the filtered object is \(\lambda^{r-1}\)-divisible but not \(\lambda^r\)-divisible. The finite-stage calculations of the \(S_{K(1)}\)-cohomology of stunted lens spectra determine this divisibility and therefore force all the differentials in Theorem~\ref{thm:intro-ahss}. 

 \begin{maintheorem}\label{thm:intro-ahss} 
 Let $q=2(p-1)$. For every \(n\in\mathbb Z\), all differentials and permanent cycles in the cohomological Atiyah--Hirzebruch spectral sequence computing \[ S_{K(1)}^*(P_{2n+1}) \] are determined. The nonzero higher differentials occur in two families, whose lengths are governed by \(p\)-adic valuations and are of the form \[ \left(v_p(k)+1\right)q \qquad\text{and}\qquad \left(v_p(k)+1\right)q+1. \] A complete description of their sources, targets, and surviving classes is given in \cref{BCPAHSS}.
\end{maintheorem}  
This spectral-sequence calculation recovers the odd-weight part of Theorem~\ref{thm:intro-additive}, but it also gives more refined information: it identifies the precise filtration in which each class appears and explains the differential lengths through a divisibility criterion. 

\section*{Acknowledgment}
The author is deeply grateful to Mark Behrens for all the helpful discussions and especially for explaining the computations of $\K(1)$-local Toda brackets. The author would also like to thank Sihao Ma, Shangjie Zhang and Yingxin Li for helpful conversations.

\section*{Notation}
Throughout the paper, we use the following notation:
\begin{itemize}
    \item For $X \in \Sp$, we write $b(X)$ to denote the $C_p$-genuine spectrum $F({EC_p}_+, i_*X)$.
    \item We fix $p$ as an odd prime.
    \item We fix $g$ to be an integer such that it is primitive modulo $p^2$, which means it is a topological generator of $\mathbb{Z}_p^{\times}$.
    \item We fix $\omega$ to be a primitive $(p-1)$-th root of unity in $\mathbb{Z}_p$ such that $g \equiv \omega$ mod $p$.
    \item Let $a\in \mathbb{Z}_p^{\times}$. We use $\sphere_{\K(1)}[a]$ to denote a $\K(1)$-local spectrum in the Picard group $\Pic_1$ of $\K(1)$-local spectra, characterized by the fact that $\pi_* (\KU_p \otimes\sphere_{\K(1)}[a])^{\wedge}_p\cong \pi_*{\KU_p}$ and the Adams operation $\psi^g$ acting on $\pi_0 (\KU_p \otimes\sphere_{\K(1)}[a])^{\wedge}_p \cong \mathbb{Z}_p$ as the scalar multiplication by $a$ \cite[Proposition 2.1]{hopkins1994constructions}.
    \item  We fix $q=2(p-1)$ as the degree of $v_1$.
    \item We use $P_j$ to denote the spectrum $\Sigma^{\infty}(BC_p)_j^{\infty}$.
    \item We adopt the conventions $v_p(0)=\infty$ and $\mathbb{Z}/p^0=0$.
 \end{itemize}

\section{Preliminaries and Equivariant Background}

We recall some facts about $C_p$-equivariant homotopy theory and fix notation. Genuine equivariant homotopy groups are typically $RO(G)$-graded. The group $C_p$ has one trivial representation and $(p-1)/2$ two-dimensional irreducible representations coming from rotations on $\mathbb{C}$ by non-trivial $p$-th roots of unity. Because the non-trivial representation spheres are all equivalent $p$-locally (e.g. \cite{HHR17}), we can use a two-parameter grading generated by $1$ and $\lambda$. 

The motivation to introduce the spoke grading is that an isomorphism on $RO(G)$-graded homotopy groups is not enough to determine a weak equivalence between genuine $G$-spectra if $G\neq C_2$. To fix this issue, we have the following definition

\begin{defn} \cite{HSW23}
    The spoke sphere $\sphere^{\sigma}$ is defined by the cofiber sequence
    \begin{equation*}
        {C_p}_+ \rightarrow \sphere_{C_p} \stackrel{a_{\sigma}}{\longrightarrow} \sphere^{\sigma}
    \end{equation*} 
    We also define $S^{-\sigma}$ as its Spanier-Whitehead dual in $C_p$-genuine spectra.
\end{defn}
The spoke sphere has the following basic properties:
\begin{prop} \label{Ssigma}From \cite[section 5]{HSW23} and \cite[A.2]{angelini2025spoke}, there are noncanonical equivalences
\begin{equation*}
    \sphere^{\sigma} \wedge \sphere^{-\sigma}\cong \sphere \oplus \bigvee_{p-2} \ {C_p}_+, \quad \sphere^{\lambda} \wedge \sphere^{-\sigma} \cong \sphere^{\sigma}. 
\end{equation*}
These equivalences imply 
\begin{equation*}
\sphere^{\sigma} \wedge \sphere^{\sigma}\cong \sphere^{\lambda} \oplus \bigvee_{p-2} \Sigma^2 {C_p}_+
\end{equation*}
\end{prop}
With the spoke spheres, we adopt the following notation
\begin{equation} \label{defofspoke}
    \pi^{C_p}_{i,j}(X) \colon=
    \left\{  
             \begin{aligned}
                &\pi_0 \Map_{C_p}(\sphere^{i-j+k\lambda} , X), & \quad j=2k \\
                 &\pi_0 \Map_{C_p}(\sphere^{i-j+k\lambda+\sigma} , X), &\quad j=2k+1 \\
             \end{aligned}   
\right.
\end{equation}
The spoke-graded homotopy groups are sometimes called the $\ROS(G)$-graded homotopy groups.
A crucial feature of the spoke grading is that it is enough to detect weak equivalences:
\begin{prop} \cite[Proposition $3.1$]{hou2025c3equivariantstablestems}
    A map $f: X\rightarrow Y$ in $\Sp^{C_p}$ is a weak equivalence if and only if it induces an isomorphism on the spoke graded homotopy groups.
\end{prop}

For $Y \in \Sp$, we define $b(Y)$ to be the Borel-complete genuine $C_p$-spectrum $F({EC_p}_+ , i_*Y)$, where $i_*$ is the inflation functor from $\Sp$ to $\Sp^{C_p}$. The spoke-graded homotopy groups of $b(Y)$ are as follows:
\begin{prop} \label{spoketocohomo}
    We have an isomorphism
    \begin{equation*}
        \pi_{i,j} b(Y) \cong Y^{j-i}(P_j),
    \end{equation*}
    where $P_j = \Sigma^{\infty} ({BC_p})_{j}^{\infty}$.
\end{prop}

\begin{proof}
    The proof is similar to the proof of \cite[Proposition 3.2]{hou2025c3equivariantstablestems}. When $j=2k$, we have the following isomorphism 
    \begin{equation*}
\begin{aligned}
     \pi_{i,j}(b(Y)) & \cong \pi_0 \Map_{\Sp_{C_p}}(\sphere^{i-j+k\lambda}, F({EC_p}_+, i_*Y)) \cong \pi_0 \Map_{\Sp_{C_p}}({EC_p}_+ \wedge \sphere^{i-j+k\lambda},  i_*Y)  \\
     & \cong \pi_0 \Map_{\Sp}(({EC_p}_+ \wedge\sphere^{i-j+k\lambda})_{C_p},  Y) \cong \pi_0 \Map_{\Sp}( (\sphere^{i-j+k\lambda})_{hC_p},  Y)
\end{aligned}  
\end{equation*}
By \cite[Chapter $V$, Theorem 2.4]{bruner2006h}, we have
\begin{equation*}
    (\sphere^{i-j+k\lambda})_{hC_p} \cong (BC_p)^{i-j+k\lambda} \cong \Sigma^{i-j} P_{2k}.
\end{equation*}
This tells us that
    $\pi_{i,j}(b(Y)) \cong Y^{j-i}(P_{j})$.
When $j=2k+1$, the odd case is identical once we establish $P_{2k+1} \cong (\sphere^{k\lambda+ \sigma})^{hC_p}$.
The above follows from the cofiber sequence
\begin{equation*}
    (\sphere^{k\lambda} \otimes {C_p}_+)^{hC_p} \rightarrow (\sphere^{k\lambda})^{hC_p} \rightarrow (\sphere^{k\lambda+ \sigma})^{hC_p}
\end{equation*}
and the first map can be seen as the bottom cell inclusion when identifying $(\sphere^{k\lambda})^{hC_p}$ as the Thom spectrum.
\end{proof}

\section{Odd tails and $\K(1)$-local decomposition of $P_{2n+1}$}

\begin{lemma}\label{characterPicard}
    Let $X$ be a spectrum such that $\KU_p^*(X) \cong \Sigma^{2k}\KU_p^*$ as $\KU_p^*$-modules for some $k\in \mathbb{Z}$. Suppose that $\psi^g$ acts on $\KU_p^0(X)$ as the scalar multiplication by $a \in \mathbb{Z}_p^{\times}$. Then $L_{\K(1)}X$ is invertible, and
    \begin{equation*}
        F(X, \sphere_{\K(1)}) \cong \sphere_{\K(1)}[a], \quad L_{\K(1)}X \cong \sphere_{\K(1)}[a^{-1}].
    \end{equation*}
\end{lemma}

\begin{proof}
  Because $\K(1)$ is a summand of $\KU_p/p$, we learn that 
  \begin{equation*}
      \dim_{\K(1)_*} \K(1)_*(X)=   \dim_{\K(1)^*} \K(1)^*(X) =1.
  \end{equation*}
  By \cite[Theorem 1.3]{hopkins1994constructions}, we already see that $L_{\K(1)}X$ is invertible. To identify it as $\sphere_{\K(1)}[a^{-1}]$, we recall that there is a short exact sequence \cite[Proposition 2.1]{hopkins1994constructions}:
  \begin{equation*}
      0\rightarrow M \rightarrow \pi_0\operatorname{Pic}(\Sp_{\K(1)}) \rightarrow \mathbb{Z}/2 \rightarrow 0.
  \end{equation*}
  The kernel $M$ is identified via the isomorphism $\operatorname{ev}: M \xrightarrow[]{\cong} \mathbb{Z}_p^{\times}$ where $\operatorname{ev}(X)$ is the eigenvalue of $\psi^g$ on $\lim_n\KU_p(X\otimes \sphere/p^n)$. Because $L_{\K(1)}X$ is invertible, we have
  \begin{equation*}
      \KU_p \otimes F(X,\sphere_{\K(1)}) \cong \KU_p \otimes F(L_{\K(1)}X, \sphere_{\K(1)}) \cong F(L_{\K(1)}X, \KU_p) \cong F(X, \KU_p) 
  \end{equation*}
 By assumption, $F(X, \KU_p)$ is a finite-rank free $\KU_p$-module, so we have
 \begin{equation*}
     \lim_n \KU_p \otimes F(X,\sphere_{\K(1)}) \otimes \sphere/p^n \cong \KU_p \otimes F(X,\sphere_{\K(1)}) \cong F(X, \KU_p)
 \end{equation*}
Hence, the eigenvalue isomorphism identifies $F(X,\sphere_{\K(1)})$ with $\sphere_{\K(1)}[a]$. Hence $L_{\K(1)}X$ is isomorphic to $ \sphere_{\K(1)}[a^{-1}]$ as the dual.
\end{proof}

\begin{prop} \label{fixpoint}
     Consider $\sphere_{\K(1)}$ with the trivial $C_p$-action and let $\omega$ be a primitive $(p-1)$-th root of unity in $\mathbb{Z}_p$. We have an isomorphism of spectra 
    \begin{equation*}
  \sphere_{\K(1)}^{hC_p} \cong \sphere_{\K(1)} \oplus \bigvee_{i=1}^{p-1} \sphere_{\K(1)}[{\omega^i}].
    \end{equation*}
\end{prop}

\begin{proof}
 Let us quickly recall the construction of the stable splitting of $\Sigma^{\infty} BC_p $\cite{harris1988stable}: 
    Because $\Sigma^{\infty} BC_p$ is $p$-complete and any element in $\mathbb{F}_p^{\times}=\langle \sigma \rangle$ induces an automorphism of $BC_p$, we see there is a ring map as follows:
    \begin{equation} \label{ringmap}
     \mathbb{Z}_p[\mathbb{F}_p^{\times}] \rightarrow \pi_0 \Map(\Sigma^{\infty} BC_p, \Sigma^{\infty} BC_p),
    \end{equation}
    In the ring $ \mathbb{Z}_p[\mathbb{F}_p^{\times}]$, we can write $1$ as 
    \begin{equation*}
        1=e_1+e_2+ \cdots + e_{p-1},
    \end{equation*}
    where 
    \begin{equation*}
        e_i= \frac{1}{p-1} \sum_{n=0}^{p-2} \omega^{-in}\sigma^n \in \mathbb{Z}_p[\langle \sigma \rangle].
    \end{equation*}
    This idempotent decomposition in the ring induces a decomposition of $\Sigma^{\infty}BC_p$ as
    \begin{equation*}
        \Sigma^{\infty}BC_p \cong X_1 \vee X_2 \vee \cdots \vee X_{p-1}.
    \end{equation*}
    In particular, we also know that $X_{p-1} \cong \Sigma^{\infty} B\Sigma_p$ \cite[Theorem B]{harris1988stable}. The stable decomposition implies that
    \begin{equation*}
         \sphere_{\K(1)}^{hC_p} \cong 
        F(\Sigma^{\infty}_+ BC_p, \sphere_{\K(1)}) \cong \sphere_{\K(1)} \oplus \bigvee_{i=1}^{p-1} \Map(X_i, \sphere_{\K(1)}).
    \end{equation*}
We will identify these $\K(1)$-local summands by \cref{characterPicard}. Notice the ring map (\ref{ringmap}) can be extended to a map as follows: 
\begin{equation*}
        \mathbb{Z}_p[\mathbb{F}_p^{\times}] \rightarrow \pi_0 \Map(\Sigma^{\infty} BC_p, \Sigma^{\infty} BC_p) \rightarrow  \operatorname{Hom}_{\mathbb{Z}_p}( \widetilde{\KU}_p^0({BC_p}), \widetilde{\KU}_p^0({BC_p})).
\end{equation*}
   By the Atiyah-Segal completion theorem, we know that 
   \begin{equation*}
        \KU_p^0({BC_p}) \cong \operatorname{R}_{\mathbb{C}}(C_p)^{\wedge}_p \cong \mathbb{Z}_p[t]/(t^p-1).
   \end{equation*}
    Under this isomorphism, $\sigma \in \mathbb{F}_p^{\times}$ induces a ring automorphism of $\KU_p^0({BC_p}_+)$ given by $t \mapsto t^g$, where $g$ is primitive modulo $p^2$ and also primitive modulo $p$. Under the decomposition
    \begin{equation*}
         \KU_p^0({BC_p}) \cong \KU_p^0(*) \oplus \widetilde{\KU}_p^0(BC_p),
    \end{equation*}
    the idempotent $e_i$ induces a map on $\widetilde{\KU}_p^0(BC_p)\cong \Ker(\epsilon)$, where $\epsilon\colon \KU_p^0({BC_p})  \rightarrow \KU_p^0(*)$ is the augmentation map. Consider a natural basis $\{t^{g^n}-1\}_{n=1}^{p-1}$ on the kernel and perform the discrete Fourier transform, then we get another basis:
    \begin{equation*}
    \begin{aligned}
         &v_{p-1}=\frac{1}{p-1} \sum_{n=0}^{p-2} (t^{g^n} -1) ;\\
         &v_i=\frac{1}{p-1} \sum_{n=0}^{p-2} \omega^{-in}t^{g^n} \qquad \text{for } i=1, \cdots, p-2. \\
    \end{aligned}    
    \end{equation*}
    By noticing that $\sigma(v_i)=\omega^i v_i$, we have
    \begin{equation*}
    \begin{aligned}
        e_i(v_i) &=\frac{1}{p-1} \sum_{n=0}^{p-2} \omega^{-in}\sigma^n(v_i)
          = \frac{1}{p-1} \sum_{n=0}^{p-2} \omega^{-in} \omega^{in} v_i 
          = v_i; \\
        e_i(v_j) &=\frac{1}{p-1} \sum_{n=0}^{p-2} \omega^{-in}\sigma^n(v_j)
          = \frac{1}{p-1} \sum_{n=0}^{p-2} \omega^{-in} \omega^{jn} v_j 
          = 0.\\
    \end{aligned} 
    \end{equation*}
    This shows $X_i=e_i\Sigma^{\infty}BC_p$ has $\KU_p$-cohomology
    \begin{equation*}
        \KU_p^0(X_i) \cong \mathbb{Z}_p\{v_i\}
    \end{equation*}
    as a $\mathbb{Z}_p^{\times}$-submodule of $\widetilde{\KU}_p^0({BC_p})$, so we have $\KU_p^*(X_i) \cong \KU_p^*$ and the Adams operation $\psi^g$ acts as
    $$\psi^g(v_i)=\omega^i v_i.$$
 \cref{characterPicard} therefore gives
        $ \Map(X_i, \sphere_{\K(1)}) \cong \sphere_{\K(1)}[\omega^i]$, which finishes the proof.
\end{proof}
As an immediate corollary, we also have
    \begin{corollary}     
 \label{k1bcp}
         Consider the stable decomposition $\Sigma^{\infty}_{+} BC_p \cong \sphere \oplus \bigvee_{i=1}^{p-1}X_i $. We have $L_{\K(1)}X_i \cong\sphere_{\K(1)}[{\omega^{-i}}] $, which implies
        \begin{equation*}
           L_{\K(1)} \Sigma^{\infty}_{+} BC_p \simeq \sphere_{\K(1)} \oplus \bigvee_{i=1}^{p-1} \sphere_{\K(1)}[{\omega^i}],
        \end{equation*}
        where $\omega$ is the chosen primitive $(p-1)$-th root of unity in $\mathbb{Z}_p$. 
   \end{corollary} 
   
We first treat $n\geq 0$:

\begin{prop}\label{decomposeatpositive}
    For any $n \geq 0$ and $1 \leq i \leq p-1$, we have $L_{\K(1)}(X_i)_{2n+1}^{\infty} \cong \sphere_{\K(1)}[\omega^{-i}]$, and
    \begin{equation*}
        L_{\K(1)} P_{2n+1} \simeq  \bigvee_{i=1}^{p-1} \sphere_{\K(1)}[{\omega^{i}}].
    \end{equation*}
\end{prop}
\begin{proof}
    By cellular approximation, we can make the idempotent $e_j$ cellular with respect to the minimal cell structure. Notice that $e_j H^*(BC_p; \mathbb{F}_p) \neq 0$ iff $* \equiv 2j, 2j-1$ mod $2(p-1)$. This tells us that the $p$-localized $X_j$ has the minimal cell structure with one cell at each dimension $* \equiv 2j, 2j-1$ mod $2(p-1)$, and 
   \begin{equation} \label{Pndecom}
    P_n \simeq \bigvee_{i=1}^{p-1} (X_i)_{n}^{\infty}.
   \end{equation}
   By \cite[Chapter $V$, Theorem 2.4]{bruner2006h}, we know $P_{2n} \simeq (BC_p)^{n\lambda}$, where $\lambda$ is a rotation $C_p$-representation. The Thom isomorphism gives $\KU_p^*(P_{2n}) \cong \KU_p^{*-2n}({BC_p}_+)$. Then we can use the cofiber sequence $\sphere^{2n} \rightarrow P_{2n} \rightarrow P_{2n+1}$ to learn that
   \begin{equation*}
       \KU_p^*(P_{2n+1}) \cong
           \left\{  
             \begin{aligned}
                &\bigoplus_{i=1}^{p-1} \mathbb{Z}_p, & \quad *=2k \\
                 &0, &\quad *=2k+1 \\
             \end{aligned}   
\right.
   \end{equation*}
   Consider the long exact sequence induced by the cofiber sequence $\Sigma^{\infty} (BC_p)^{2n} \rightarrow \Sigma^{\infty} BC_p \rightarrow P_{2n+1}$. Combining the above calculation and the calculation of $\KU_p^*((BC_p)^{2n})$ from \cite[Chapter $V$, Theorem 2.3]{bruner2006h}, we know the long exact sequence breaks into short exact sequences in each degree
   \begin{equation*}
       0 \rightarrow \bigoplus_{i=1}^{p-1} \KU_p^*((X_i)_{2n+1}^{\infty}) \rightarrow \bigoplus_{i=1}^{p-1} \KU_p^*(X_i) \rightarrow \KU_p^*((BC_p)^{2n}) \rightarrow 0, 
   \end{equation*}
   which realizes $\KU_p^0((X_i)_{2n+1}^{\infty})\cong \mathbb{Z}_p$ as a $\mathbb{Z}_p^{\times}$-submodule of $\KU_p^0(X_i)$. Therefore $\psi^g$ acts on $\KU_p^0((X_i)_{2n+1}^{\infty})$ as the scalar multiplication by $\omega^i$. Again, by \cref{characterPicard} we see that we must have
   \begin{equation*}
       L_{\K(1)} (X_i)_{2n+1}^{\infty} \cong \sphere_{\K(1)}[\omega^{-i}].
   \end{equation*}
\end{proof}

For the more general case where $n\in \mathbb{Z}$, we recall that, under James periodicity, $P_{2n+\epsilon}^{2n+k}$ for $\epsilon=0,1$ and $k >0$ can be constructed as $\Sigma^{2(n-r)} P_{2r+\epsilon}^{2r+k}$ for any $r \equiv n \mod{p^{\lfloor k/q \rfloor}}$, and $(X_i)_{2n+\epsilon}^{2n+k}$ can be defined similarly for $n <0$.
We have

\begin{prop} \label{decomposeatnegative}
    For $n <0$ and $1 \leq i \leq p-1 $, we have $L_{\K(1)}(X_i)_{2n+1}^{\infty} \cong \sphere_{\K(1)}[\omega^{-i}]$, and
        \begin{equation*}
        L_{\K(1)} P_{2n+1} \simeq  \bigvee_{i=1}^{p-1} \sphere_{\K(1)}[{\omega^{i}}].
    \end{equation*}
\end{prop}
\begin{proof}
    From the description above the proposition, we see immediately that $P_n^{n+k}$ can be decomposed as $\bigvee_{i=1}^{p-1} (X_i)_{n}^{n+k}$ $p$-locally. Passing to the colimit, we get a stable decomposition 
    for $P_{n}$ as in the case of positive integers. Then we can use a similar proof as above, except that here we use induction with cofiber sequence $\Sigma^{2n-1} \sphere/p \rightarrow P_{2n-1} \rightarrow P_{2n+1}$ instead of $\Sigma^{\infty} (BC_p)^{2n} \rightarrow \Sigma^{\infty} BC_p \rightarrow P_{2n+1}$, where the base case for $n=0$ is covered by \cref{k1bcp}.
\end{proof}

The Adams-eigenvalue calculation now gives the following cohomology groups:

\begin{corollary} \label{computationatodd}
    For $k \in \mathbb{Z}$, we have
    \begin{equation*}
       \pi_{i,2k+1} b(\sphere_{\K(1)}) \cong \sphere_{\K(1)}^{2k+1-i}\left(P_{2k+1}\right) \cong
          \left\{
        \begin{aligned}
           & \mathbb{Z}/p^{v_p(k-m)+1},  & \text{if } i = 2m , \text{ with } m \neq k;\\
           & \mathbb{Z}_p, & \text{ if } i= 2k, 2k+1.\\
           & 0, & \text{otherwise.}
        \end{aligned}
        \right.
    \end{equation*}
\end{corollary}

\section{Computations for the even tails}\label{sec: evencom}

To compute the spoke-graded homotopy groups, it remains to compute $\sphere_{\K(1)}$-cohomology of $P_{2n}$. Consider the cofiber sequence
\begin{equation*}
    \sphere^{2n} \rightarrow P_{2n} \rightarrow P_{2n+1}.
\end{equation*}
Using the fact that $\sphere_{\K(1)}^{2m}(P_{2n+1})=0$ when $m\neq0$ and $\sphere_{\K(1)}^{2m}(\sphere^{2n})=0$ when $m\neq n$, we can break down the long exact sequence associated to the cofiber sequence above as short exact sequences of the following types (we are assuming $n\neq 0$):
  \begin{gather*}\sphere_{\K(1)}^{-1}(\sphere^{2n}) \stackrel{(\ast)}{\longrightarrow} \mathbb{Z}_p \rightarrow \sphere_{\K(1)}^0(P_{2n}) \rightarrow  0 \rightarrow \mathbb{Z}_p \rightarrow \sphere_{\K(1)}^1(P_{2n}) \rightarrow  \sphere_{\K(1)}^1(\sphere^{2n}) \rightarrow 0;\\
0 \rightarrow \sphere_{\K(1)}^{2n}(P_{2n}) \rightarrow \mathbb{Z}_p \rightarrow \sphere_{\K(1)}^{2n+1}(P_{2n+1}) \rightarrow \sphere_{\K(1)}^{2n+1}(P_{2n}) \rightarrow \mathbb{Z}_p \rightarrow 0;\\
0 \rightarrow \sphere_{\K(1)}^{2k+1}(P_{2n+1}) \rightarrow \sphere_{\K(1)}^{2k+1}(P_{2n}) \rightarrow  \sphere_{\K(1)}^{2k+1}(\sphere^{2n}) \rightarrow 0, \quad k \in \mathbb{Z} \setminus \{0,n\}.
\end{gather*}  
It remains to determine the extension classes in these short exact sequences. The first observation we make is that the map $(\ast)$ is the zero map:
\begin{lemma}
    For any $n \in \mathbb{Z}$, the map $\sphere_{\K(1)}^{-1}(\sphere^{2n}) \longrightarrow\sphere_{\K(1)}^0(P_{2n+1}) \cong \mathbb{Z}_p$ is the zero map.
\end{lemma}
\begin{proof}
The map introduced in the lemma is induced by the boundary map $P_{2n+1} \rightarrow \Sigma \sphere^{2n}$. Let $i$ be the positive integer $1 \leq i \leq p-1$ such that $i\equiv n $ mod $p-1$. Then the boundary map factors through the projection $P_{2n+1} \rightarrow (X_i)_{2n+1}^{\infty}$ by the $p$-local cell decomposition of $P_{2n+1}$. 

Notice that $\sphere_{\K(1)}^{-1}(\sphere^{2n}) \neq 0$ if and only if $-1 \equiv n$ mod $p-1$ and $\sphere_{\K(1)}^0((X_i)_{2n+1}^{\infty}) \neq 0$ iff $0 \equiv n$ mod $p-1$. The two required congruences are incompatible, so the relevant summand is zero and the boundary map vanishes.
\end{proof}

 The following proposition will be helpful in analyzing the last two short exact sequences above: 

\begin{prop} \label{attachcell}
    Let $n$ be a nonzero integer and write $2n=p^ls$ with $(p,s)=1$ . Then we have the following commutative diagram
    % https://q.uiver.app/#q=WzAsOSxbMCwwLCJcXG1hdGhiYntTfV57Mm59Il0sWzEsMCwiUF97Mm59Il0sWzIsMCwiUF97Mm4rMX0iXSxbMCwxLCJcXG1hdGhiYntTfV57Mm59Il0sWzEsMSwiUF97Mm59XnsybisyKHAtMSkobCsxKX0iXSxbMiwxLCJQX3sybisxfV57Mm4rMihwLTEpKGwrMSl9Il0sWzAsMiwiXFxtYXRoYmJ7U31eezJufSJdLFsxLDIsIkNcXGJhcntcXGFscGhhX2x9Il0sWzIsMiwiXFxtYXRoYmJ7U31eezJuKzIocC0xKShsKzEpfSJdLFswLDFdLFsxLDJdLFswLDMsIiIsMix7ImxldmVsIjoyLCJzdHlsZSI6eyJoZWFkIjp7Im5hbWUiOiJub25lIn19fV0sWzMsNF0sWzQsNV0sWzQsMV0sWzUsMl0sWzMsNiwiIiwyLHsibGV2ZWwiOjIsInN0eWxlIjp7ImhlYWQiOnsibmFtZSI6Im5vbmUifX19XSxbNiw3XSxbNCw3XSxbNyw4XSxbNSw4XV0=
\[\begin{tikzcd}
	{\mathbb{S}^{2n}} & {P_{2n}} & {P_{2n+1}} \\
	{\mathbb{S}^{2n}} & {P_{2n}^{2n+q(l+1)}} & {P_{2n+1}^{2n+q(l+1)}} \\
	{\mathbb{S}^{2n}} & {\Sigma^{2n} C\bar{\alpha}_{l+1}} & {\mathbb{S}^{2n+q(l+1)}}
	\arrow[from=1-1, to=1-2]
	\arrow[equals, from=1-1, to=2-1]
	\arrow[from=1-2, to=1-3]
	\arrow[from=2-1, to=2-2]
	\arrow[equals, from=2-1, to=3-1]
	\arrow[from=2-2, to=1-2]
	\arrow[from=2-2, to=2-3]
	\arrow[from=2-2, to=3-2]
	\arrow[from=2-3, to=1-3]
	\arrow[from=2-3, to=3-3]
	\arrow[from=3-1, to=3-2]
	\arrow[from=3-2, to=3-3]
\end{tikzcd}\]
The horizontal sequences are cofiber sequences, and $\bar{\alpha}_{l+1}$ represents a generator of $ \pi_{(l+1)q-1}(\sphere_{\K(1)})\cong \mathbb{Z}/p^{v_p(l+1)+1}$.
\end{prop}

\begin{proof}
    The upper half of the diagram comes from the cell structure of $P_{2n}$. For the lower half part, it remains to identify the primary co-attaching map out of $\sphere^{2n}$ as $\bar{\alpha}_{l+1}$.
    
    Let $e$ be a positive integer such that $(p-1)^e \equiv 1$
    mod $p^{l+1}$. By \cite[Theorem 1.1]{gonzalez1996classification}, we have an isomorphism
    \begin{equation*}
        \Sigma^{2n(p-1)^e-2n} P_{2n}^{2n+q(l+1)} \cong P_{2n(p-1)^e}^{2n(p-1)^e+q(l+1)}.
    \end{equation*}
    Hence we can assume $2n=2k(p-1)$. By the $p$-local decomposition of $P_{2n}$, the bottom cell $\sphere^{2n}$ 
    is only attached to $X_{p-1}=\Sigma^{\infty}B\Sigma_p$, so it is enough to construct the diagram after replacing each occurrence of $P$ by $B\Sigma_p$. From \cite[Chapter $V$, Theorem $2.6$]{bruner2006h}, the bottom cell inclusion $\sphere^{2n} \rightarrow {(B\Sigma_p)_{2n}^{2n+q(l+1)-1}} $ has a left inverse, denoted by $\rho$; then we have the following commutative diagram:
    % https://q.uiver.app/#q=WzAsMTEsWzEsMCwiXFxtYXRoYmJ7U31eezJufSJdLFsyLDAsIlBfezJufV57Mm4rMihwLTEpKGwrMSl9Il0sWzMsMCwiUF97Mm4rMX1eezJuKzIocC0xKShsKzEpfSJdLFsxLDEsIlBfezJufV57Mm4rMihwLTEpKGwrMSktMX0iXSxbMiwxLCJQX3sybn1eezJuKzIocC0xKShsKzEpfSJdLFszLDEsIlxcbWF0aGJie1N9XnsybisyKHAtMSkobCsxKX0iXSxbMSwyLCJcXG1hdGhiYntTfV57Mm59Il0sWzIsMiwiWCJdLFszLDIsIlxcbWF0aGJie1N9XnsybisyKHAtMSkobCsxKX0iXSxbMCwxLCJcXG1hdGhiYntTfV57Mm4rMihwLTEpKGwrMSktMX0iXSxbMCwyLCJcXG1hdGhiYntTfV57Mm4rMihwLTEpKGwrMSktMX0iXSxbMCwxXSxbMSwyXSxbMCwzXSxbMyw0XSxbNCw1XSxbNCwxLCIiLDEseyJsZXZlbCI6Miwic3R5bGUiOnsiaGVhZCI6eyJuYW1lIjoibm9uZSJ9fX1dLFsyLDVdLFszLDYsInAiLDJdLFs2LDddLFs0LDddLFs3LDhdLFs1LDgsIiIsMix7ImxldmVsIjoyLCJzdHlsZSI6eyJoZWFkIjp7Im5hbWUiOiJub25lIn19fV0sWzksMywiXFxwYXJ0aWFsIl0sWzksMTAsIiIsMCx7ImxldmVsIjoyLCJzdHlsZSI6eyJoZWFkIjp7Im5hbWUiOiJub25lIn19fV0sWzEwLDYsInAgXFxjaXJjIFxccGFydGlhbCJdXQ==
\[\begin{tikzcd}
	& {\mathbb{S}^{2n}} & {(B\Sigma_p)_{2n}^{2n+q(l+1)}} & {(B\Sigma_p)_{2n+1}^{2n+q(l+1)}} \\
	{\mathbb{S}^{2n+q(l+1)-1}} & {(B\Sigma_p)_{2n}^{2n+q(l+1)-1}} & {(B\Sigma_p)_{2n}^{2n+q(l+1)}} & {\mathbb{S}^{2n+q(l+1)}} \\
	{\mathbb{S}^{2n+q(l+1)-1}} & {\mathbb{S}^{2n}} & X & {\mathbb{S}^{2n+q(l+1)}}
	\arrow[from=1-2, to=1-3]
	\arrow[from=1-2, to=2-2]
	\arrow[from=1-3, to=1-4]
	\arrow[from=1-4, to=2-4]
	\arrow["\partial", from=2-1, to=2-2]
	\arrow[equals, from=2-1, to=3-1]
	\arrow[from=2-2, to=2-3]
	\arrow["\rho"', from=2-2, to=3-2]
	\arrow[equals, from=2-3, to=1-3]
	\arrow[from=2-3, to=2-4]
	\arrow[from=2-3, to=3-3]
	\arrow[equals, from=2-4, to=3-4]
	\arrow["{\rho \circ \partial}", from=3-1, to=3-2]
	\arrow[from=3-2, to=3-3]
	\arrow[from=3-3, to=3-4]
\end{tikzcd}\]
We want to show that $\rho \circ \partial \simeq \bar{\alpha}_{l+1}$, and this can be obtained by dualizing \cite[Chapter $V$, Proposition $2.17$]{bruner2006h}.
\end{proof}

To understand the extension classes, we first compute the lower part of the diagram, and use it to understand the top row.

\begin{prop}
    Let $n \in \mathbb{Z}$ and $k$ be a non-negative integer. Then we have
    \begin{equation*}
        \sphere_{\K(1)}^m\left((X_i)_{nq+2i-1}^{(n+k)q+2i}\right) \cong
        \left\{
        \begin{aligned}
           & \mathbb{Z}/p^{\min\{k+1, v_p(a)+1 \} }, \quad & \text{if } m= 2a+1  \text{ or } 2a \text{ with } a \equiv i \mod{p-1};\\
          &0, &\text{otherwise.}\\
        \end{aligned}
        \right.
    \end{equation*}
\end{prop}

\begin{proof}
    First assume $n=0$, so we can omit the lower index since it is just the $\left(kq+2i\right)$-skeleton of $X_i$. From the Atiyah-Hirzebruch spectral sequence, it is easy to see that
    \begin{equation*}
        \KU_p^*\left((X_i)^{kq+2i} \right) \cong  
                \left\{
        \begin{aligned}
           & \mathbb{Z}/p^{k+1}  &*=2m;\\
          &0, &\text{otherwise.}\\
        \end{aligned}
        \right.
    \end{equation*}
    The cofiber sequence $(X_i)^{kq+2i} \rightarrow X_i \rightarrow (X_i)_{kq+2i+1}^{\infty}$ implies that the Adams operation on $\KU_p^{2a}((X_i)^{kq+2i})$ is scalar multiplication by $\omega^ig^{-a}$. 
   It is not an isomorphism if and only if $a \equiv i, \mod{p-1}$, and at this degree, we have    
    \begin{equation*}
        0\rightarrow \sphere_{\K(1)}^{2a}((X_i)^{kq+2i})\rightarrow \mathbb{Z}/p^{k+1} \stackrel{\times p^{1+v_p(a)}}{\longrightarrow}\mathbb{Z}/p^{k+1} \rightarrow \sphere_{\K(1)}^{2a+1}((X_i)^{kq+2i}) \rightarrow 0.
    \end{equation*}
    Hence, we have 
    \begin{equation*}
        \sphere_{\K(1)}^m\left((X_i)^{kq+2i}\right) \cong
         \left\{
        \begin{aligned}
           & \mathbb{Z}/p^{\min\{k+1, v_p(a)+1 \} }, \quad & \text{if } m= 2a+1  \text{ or } 2a \text{ with } a \equiv i,\mod{p-1};\\
          &0, &\text{otherwise.}\\
        \end{aligned}
        \right.
    \end{equation*} 
    For any $n \geq 0$, we just need to use the long exact sequence induced by the cofiber sequence $(X_i)^{(n-1)q+2i} \rightarrow (X_i)^{(n+k)q+2i} \rightarrow (X_i)_{nq+2i-1}^{(n+k)q+2i}$. When $n <0$, the computation is carried out by James periodicity. 
\end{proof}

\begin{corollary}\label{Xinoextension}
    Let $n \in \mathbb{Z}$ and $k$ be a non-negative integer. Using the stable decomposition of $P_{2n+1}^{2n+q(k+1)}$, we see that
    \begin{equation*}
        \sphere_{\K(1)}^m\left(P^{2n+(k+1)q}_{2n+1}\right) \cong
           \mathbb{Z}/p^{\min\{k+1, v_p(a)+1 \} }, \quad \text{for } m= 2a+1  \text{ or } 2a.\\
    \end{equation*}
\end{corollary}

We remark that a similar calculation is done in \cite{thompson1990}, where he calculated the homolog groups $J_*(( B\Sigma_p)_{nq-1}^{mq})$.

\begin{prop} \label{Mark-toda}
    Let $\bar{\alpha}_m \in \pi_{mq-1}(\sphere_{\K(1)})$ be a generator and let $\alpha_m\in \pi_{mq-1}(\sphere_{\K(1)})$ be an element of order $p$ in $\pi_{mq-1}(\sphere_{\K(1)})$ with the relation $\alpha_m =p^{v_p(m)} \bar{\alpha}_m$ ($\alpha_0$ is not defined since $\pi_{-1} (\sphere_{\K(1)})$ is torsion free). Then we have
    \begin{equation*}
        \langle \bar{\alpha}_n, \alpha_m, p\rangle=\bar{\alpha}_{m+n}+p \pi_{(n+m)q-1}(\sphere_{\K(1)}), \quad m\neq0.
    \end{equation*}
\end{prop}

\begin{proof}
Let $\iota \colon \sphere_{\K(1)} \rightarrow \sphere_{\K(1)} /p$ be the bottom cell inclusion map and $\rho \colon \sphere_{\K(1)} /p \rightarrow \Sigma \sphere_{\K(1)}$ be the top cell projection map. From the construction of the Toda bracket, we have the following diagram
% https://q.uiver.app/#q=WzAsNyxbMCwwLCJcXG1hdGhiYntTfV97XFxLKDEpfV57KG4rbSlxLTJ9Il0sWzEsMCwiXFxtYXRoYmJ7U31fe1xcSygxKX1eeyhuK20pcS0yfSJdLFsxLDEsIlxcU2lnbWFeeyhuK20pcS0yfVxcbWF0aGJie1N9X3tcXEsoMSl9L3AiXSxbMSwyLCJcXG1hdGhiYntTfV97XFxLKDEpfV57KG4rbSlxLTF9Il0sWzIsMCwiXFxtYXRoYmJ7U31fe1xcSygxKX1ee25xLTF9Il0sWzMsMCwiXFxtYXRoYmJ7U31fe1xcSygxKX0iXSxbMywyLCJcXG1hdGhiYntTfV97XFxLKDEpfS9wIl0sWzAsMSwicCJdLFsxLDIsIlxcaW90YSJdLFsyLDMsIlxccmhvIl0sWzEsNCwiXFxhbHBoYV9tIl0sWzQsNSwiXFxiYXJ7XFxhbHBoYX1fbiJdLFszLDYsIlxcaW90YSBcXGNpcmMgXFxsYW5nbGUgXFxiYXJ7XFxhbHBoYX1fbiwgXFxhbHBoYV9tLCBwXFxyYW5nbGUiLDJdLFsyLDQsIlxcd2lkZXRpbGRle1xcYWxwaGFfbSBcXGNpcmMgXFxyaG99IiwxXSxbMyw1LCJcXGxhbmdsZSBcXGJhcntcXGFscGhhfV9uLCBcXGFscGhhX20sIHBcXHJhbmdsZSIsMl0sWzUsNiwiXFxpb3RhIl1d
\[\begin{tikzcd}
	{\mathbb{S}_{\K(1)}^{(n+m)q-2}} & {\mathbb{S}_{\K(1)}^{(n+m)q-2}} & {\mathbb{S}_{\K(1)}^{nq-1}} & {\mathbb{S}_{\K(1)}} \\
	& {\mathbb{S}^{(n+m)q-2}_{\K(1)}/p} \\
	& {\mathbb{S}_{\K(1)}^{(n+m)q-1}} && {\mathbb{S}_{\K(1)}/p}
	\arrow["p", from=1-1, to=1-2]
	\arrow["{\alpha_m}", from=1-2, to=1-3]
	\arrow["\iota", from=1-2, to=2-2]
	\arrow["{\bar{\alpha}_n}", from=1-3, to=1-4]
	\arrow["\iota", from=1-4, to=3-4]
	\arrow["{\widetilde{\alpha_m \circ p}}"{description}, from=2-2, to=1-3]
	\arrow["\rho", from=2-2, to=3-2]
	\arrow["{\langle \bar{\alpha}_n, \alpha_m, p\rangle}"', from=3-2, to=1-4]
	\arrow["{\iota \circ \langle \bar{\alpha}_n, \alpha_m, p\rangle}"', from=3-2, to=3-4]
\end{tikzcd}\]
To prove the Proposition, it is enough to show $\iota \circ \langle \bar{\alpha}_n, \alpha_m, p\rangle \neq 0$. Consider the following diagram 
% https://q.uiver.app/#q=WzAsNSxbMSwwLCJcXHBpXzBcXG9wZXJhdG9ybmFtZXtNYXB9KFxcbWF0aGJie1N9XntucS0yfV97XFxLKDEpfS9wLFxcbWF0aGJie1N9XntucS0xfV97XFxLKDEpfSkgIl0sWzEsMSwiXFxwaV8wXFxvcGVyYXRvcm5hbWV7TWFwfShcXG1hdGhiYntTfV57KG4rbSlxLTJ9X3tcXEsoMSl9L3AsXFxtYXRoYmJ7U31ee25xLTF9X3tcXEsoMSl9KSJdLFsyLDEsIlxccGlfMFxcb3BlcmF0b3JuYW1le01hcH0oXFxtYXRoYmJ7U31eeyhuK20pcS0yfV97XFxLKDEpfSxcXG1hdGhiYntTfV57bnEtMX1fe1xcSygxKX0pIl0sWzMsMSwiXFxwaV8wXFxvcGVyYXRvcm5hbWV7TWFwfShcXG1hdGhiYntTfV57KG4rbSlxLTJ9X3tcXEsoMSl9LFxcbWF0aGJie1N9XntucS0xfV97XFxLKDEpfSkiXSxbMCwxLCIwIl0sWzAsMSwiKHZfMV5tKV4qIiwyXSxbMSwyLCJcXGlvdGFeKiJdLFsyLDMsInBeKiJdLFs0LDFdXQ==
\[\begin{tikzcd}
	& {\pi_0\operatorname{Map}(\mathbb{S}^{nq-2}_{\K(1)}/p,\mathbb{S}^{nq-1}_{\K(1)}) } && \\
	0 & {\pi_0\operatorname{Map}(\mathbb{S}^{(n+m)q-2}_{\K(1)}/p,\mathbb{S}^{nq-1}_{\K(1)})} & {\pi_0\operatorname{Map}(\mathbb{S}^{(n+m)q-2}_{\K(1)},\mathbb{S}^{nq-1}_{\K(1)})} & {\pi_0\operatorname{Map}(\mathbb{S}^{(n+m)q-2}_{\K(1)},\mathbb{S}^{nq-1}_{\K(1)})}
	\arrow["{(v_1^m)^*}"', "{\cong}",from=1-2, to=2-2]
	\arrow[from=2-1, to=2-2]
	\arrow["{\iota^*}", from=2-2, to=2-3]
	\arrow["{p^*}", from=2-3, to=2-4]
\end{tikzcd}\] 
We see that one can choose a null-homotopy of $\alpha_m \circ p$ such that the following diagram commutes:
% https://q.uiver.app/#q=WzAsNSxbMCwwLCJcXG1hdGhiYntTfV97XFxLKDEpfV57KG4rbSlxLTJ9Il0sWzEsMCwiXFxtYXRoYmJ7U31fe1xcSygxKX1eeyhuK20pcS0yfSJdLFsxLDEsIlxcbWF0aGJie1N9XnsobittKXEtMn1fe1xcSygxKX0vcCJdLFsyLDAsIlxcbWF0aGJie1N9X3tcXEsoMSl9XntucS0xfSJdLFsyLDEsIlxcbWF0aGJie1N9XntucS0yfV97XFxLKDEpfS9wIl0sWzAsMSwicCJdLFsxLDIsIlxcaW90YSJdLFsxLDMsIlxcYWxwaGFfbSJdLFsyLDMsIlxcd2lkZXRpbGRle1xcYWxwaGFfbSBcXGNpcmMgXFxyaG99IiwxXSxbMiw0LCJ2XzFebSJdLFs0LDMsIlxccmhvIiwyXV0=
\[\begin{tikzcd}
	{\mathbb{S}_{\K(1)}^{(n+m)q-2}} & {\mathbb{S}_{\K(1)}^{(n+m)q-2}} & {\mathbb{S}_{\K(1)}^{nq-1}} \\
	& {\mathbb{S}^{(n+m)q-2}_{\K(1)}/p} & {\mathbb{S}^{nq-2}_{\K(1)}/p}
	\arrow["p", from=1-1, to=1-2]
	\arrow["{\alpha_m}", from=1-2, to=1-3]
	\arrow["\iota", from=1-2, to=2-2]
	\arrow["{\widetilde{\alpha_m \circ \rho}}"{description}, from=2-2, to=1-3]
	\arrow["{v_1^m}", from=2-2, to=2-3]
	\arrow["\rho"', from=2-3, to=1-3]
\end{tikzcd}\]
Recall that 
$\pi_{*} \sphere_{\K(1)}/p \cong \wedge_{\mathbb{F}_p}[\delta] \otimes \mathbb{F}_p[v_1^{\pm1}]$, where $|\delta|=-1$. The element $\delta$ can be constructed as $\iota \circ\zeta$ where $\zeta$ is a generator of $\pi_{-1}\sphere_{\K(1)}\cong\mathbb{Z}_p$ (notice that $\zeta$ is $\bar{\alpha}_0$ in our notation). From the relation $\bar{\alpha}_n \circ \iota \neq 0$, we can make a choice of $\zeta$ such that the following diagram commutes:
% https://q.uiver.app/#q=WzAsNCxbMCwwLCJcXG1hdGhiYntTfV97XFxLKDEpfV57bnEtMX0iXSxbMSwwLCJcXG1hdGhiYntTfV97XFxLKDEpfSJdLFswLDEsIlxcbWF0aGJie1N9X3tcXEsoMSl9XntucX0vcCJdLFsxLDEsIlxcbWF0aGJie1N9X3tcXEsoMSl9L3AiXSxbMCwxLCJcXGJhcntcXGFscGhhfV9uIl0sWzAsMiwiXFxkZWx0YSIsMl0sWzIsMywidl8xXm0iLDJdLFsxLDMsIlxcaW90YSJdXQ==
\[\begin{tikzcd}
	{\mathbb{S}_{\K(1)}^{nq-1}} & {\mathbb{S}_{\K(1)}} \\
	{\mathbb{S}_{\K(1)}^{nq}/p} & {\mathbb{S}_{\K(1)}/p}
	\arrow["{\bar{\alpha}_n}", from=1-1, to=1-2]
	\arrow["\delta"', from=1-1, to=2-1]
	\arrow["\iota", from=1-2, to=2-2]
	\arrow["{v_1^n}"', from=2-1, to=2-2]
\end{tikzcd}\]

As a ring, we have an isomorphism $
\pi_* \Map(\sphere_{\K(1)}/p, \sphere_{\K(1)}/p) \cong \wedge_{\mathbb{F}_p}[\zeta, \beta] \otimes \mathbb{F}_p[v_1^{\pm}]$, where $|\zeta|=|\beta|=-1$. The element $\zeta$ comes from the module structure of $\sphere_{\K(1)}/p$ over $\sphere_{\K(1)}$, and the element $\beta= \iota \circ \rho$. Hence, we have the following commutative diagram:
% https://q.uiver.app/#q=WzAsNSxbMCwxLCJcXG1hdGhiYntTfV97XFxLKDEpfV57bnEtMn0vcCJdLFsxLDAsIlxcbWF0aGJie1N9X3tcXEsoMSl9XntucS0xfSJdLFsyLDAsIlxcbWF0aGJie1N9X3tcXEsoMSl9XntucX0vcCJdLFsxLDEsIlxcbWF0aGJie1N9X3tcXEsoMSl9XntucX0iXSxbMCwyLCJcXG1hdGhiYntTfV97XFxLKDEpfV57bnEtMX0iXSxbMCwxLCJcXHJobyJdLFsxLDIsIlxcZGVsdGEiXSxbMSwzLCJcXHpldGEiLDJdLFszLDIsIlxcaW90YSJdLFswLDQsIlxcemV0YSIsMl0sWzQsMywiXFxyaG8iLDJdLFs0LDIsIlxcYmV0YSIsMix7ImN1cnZlIjozfV1d
\[\begin{tikzcd}
	& {\mathbb{S}_{\K(1)}^{nq-1}} & {\mathbb{S}_{\K(1)}^{nq}/p} \\
	{\mathbb{S}_{\K(1)}^{nq-2}/p} & {\mathbb{S}_{\K(1)}^{nq}} \\
	{\mathbb{S}_{\K(1)}^{nq-1}/p}
	\arrow["\delta", from=1-2, to=1-3]
	\arrow["\zeta"', from=1-2, to=2-2]
	\arrow["\rho", from=2-1, to=1-2]
	\arrow["\zeta"', from=2-1, to=3-1]
	\arrow["\iota", from=2-2, to=1-3]
	\arrow["\beta"', curve={height=18pt}, from=3-1, to=1-3]
	\arrow["\rho"', from=3-1, to=2-2]
\end{tikzcd}\]
Combining the preceding three diagrams, we obtain the following diagram
% https://q.uiver.app/#q=WzAsNSxbMCwxLCJcXG1hdGhiYntTfV97XFxLKDEpfV57KG4rbSlxLTF9Il0sWzIsMSwiXFxtYXRoYmJ7U31fe1xcSygxKX0vcCJdLFswLDAsIlxcbWF0aGJie1N9X3tcXEsoMSl9XnsobittKXEtMn0vcCJdLFsxLDAsIlxcbWF0aGJie1N9X3tcXEsoMSl9XntucS0yfS9wIl0sWzIsMCwiXFxtYXRoYmJ7U31fe1xcSygxKX1ee25xLTJ9L3AiXSxbMCwxLCJcXGlvdGEgXFxjaXJjIFxcbGFuZ2xlIFxcYmFye1xcYWxwaGF9X24sIFxcYWxwaGFfbSwgcFxccmFuZ2xlIl0sWzIsMCwiXFxyaG8iXSxbMiwzLCJ2XzFebSIsMl0sWzMsNCwiXFxiZXRhXFx6ZXRhIiwyXSxbNCwxLCJ2XzFebiIsMl1d
\[\begin{tikzcd}
	{\mathbb{S}_{\K(1)}^{(n+m)q-2}/p} & {\mathbb{S}_{\K(1)}^{nq-2}/p} & {\mathbb{S}_{\K(1)}^{nq-2}/p} \\
	{\mathbb{S}_{\K(1)}^{(n+m)q-1}} && {\mathbb{S}_{\K(1)}/p}
	\arrow["{v_1^m}"', from=1-1, to=1-2]
	\arrow["\rho", from=1-1, to=2-1]
	\arrow["{\beta\zeta}"', from=1-2, to=1-3]
	\arrow["{v_1^n}"', from=1-3, to=2-3]
	\arrow["{\iota \circ \langle \bar{\alpha}_n, \alpha_m, p\rangle}", from=2-1, to=2-3]
\end{tikzcd}\]
We finish the proof by observing that $\beta\zeta v_1^{n+m} \neq0$.
\end{proof}
This Toda bracket computation will be the main tool for solving extension problems. To begin, we compute the $\sphere_{\K(1)}$-cohomology of $C\bar{\alpha}_n$:
\begin{prop}\label{calphal}
  Let $n \geq 1$. We have the following isomorphism
    \begin{equation*}
          \sphere_{\K(1)}^*\left( C\bar{\alpha}_n\right) \cong
        \left\{
        \begin{aligned}
           & \mathbb{Z}_p, \quad &*=0, 1, 2n(p-1), 2n(p-1)+1;\\
          &  \mathbb{Z}/p^{v_p(m)+v_p(n+m)+2},  &*=2(m+n)(p-1)+1, \quad m \neq 0,-n ;\\
          &0, &\text{otherwise.}\\
        \end{aligned}
        \right.
    \end{equation*}
\end{prop}

\begin{proof}
  The cofiber sequence $\sphere_{K(1)}^{nq-1} \stackrel{\bar{\alpha}_n}{\longrightarrow}\sphere_{K(1)} \rightarrow C\bar{\alpha}_n$ induces a long exact sequence that breaks into the following short exact sequences $(m \in \mathbb{Z}\backslash \{0,-n\})$:
    \begin{gather*}
0 \rightarrow \mathbb{Z}_p \rightarrow\sphere_{\K(1)}^{nq}(C\bar{\alpha}_n) \rightarrow  0;\\
0 \rightarrow \sphere_{\K(1)}^{0}(C\bar{\alpha}_n) \rightarrow \mathbb{Z}_p \doublearrow{} \mathbb{Z}/p^{v_p(n)+1} \stackrel{0}{\longrightarrow} \sphere_{\K(1)}^{1}(C\bar{\alpha}_n) \rightarrow \mathbb{Z}_p \rightarrow 0;\\
0 \rightarrow  \mathbb{Z}_p\{\bar{\alpha}_0\}\rightarrow \sphere_{\K(1)}^{nq+1}(C\bar{\alpha}_n) \rightarrow \mathbb{Z}/p^{v_p(n)+1}\{\bar{\alpha}_{-n}\} \rightarrow 0;\\
0 \rightarrow \mathbb{Z}/p^{v_p(m)+1}\{\bar{\alpha}_{-m}\} \rightarrow\sphere_{\K(1)}^{(m+n)q+1}(C\bar{\alpha}_n) \rightarrow \mathbb{Z}/p^{v_p(m+n)+1}\{\bar{\alpha}_{-m-n}\} \rightarrow 0.
\end{gather*}
We need to determine the extension class of the last two short exact sequences. This is equivalent to asking what $p \bar{\alpha}_{-m-n} \in \sphere_{\K(1)}^*(C\bar{\alpha}_n)$ is. Unwinding the definition of the Toda bracket (or by comparing with the proof of \cite[Proposition 4.3]{hou2025c3equivariantstablestems}), this is equivalent to asking for $\langle \bar{\alpha}_n, \alpha_{-m-n}, p \rangle$. From \cref{Mark-toda}, we see immediately that this is $\bar{\alpha}_{-m}$, and this finishes the proof.
\end{proof}

For ease of reference, we recall the correspondence between $\operatorname{Ext}^1$ group and isomorphism classes of short exact sequences:

\begin{remark}
    Let $a, b$ be positive integers, and consider an extension of the form
    \begin{equation*}
        0 \rightarrow \mathbb{Z}/p^a \xrightarrow[]{i} E \rightarrow \mathbb{Z}/p^b \rightarrow 0.
    \end{equation*}
    We let $x=i(1)$ and $y$ be a lift of $1$ in $\mathbb{Z}/p^b$. Then there exists $\gamma$ such that $\gamma x=p^b y$. Changing the lift $y$ changes $\gamma$ by an element of $p^b(\mathbb{Z}/p^a)$; Therefore $[\gamma]$ is well-defined as a class in the image of $i$ modulo $p^b\mathbb{Z}/p^a$. This determines the isomorphism class of the short exact sequence as
    \begin{equation*}
        [\gamma] \in \frac{\mathbb{Z}/p^a}{p^b \mathbb{Z}/p^a} \cong \mathbb{Z}/p^{\min\{a,b\}} \cong \operatorname{Ext}^1(\mathbb{Z}/p^b, \mathbb{Z}/p^a).
    \end{equation*}
    We write $\gamma= p^c \theta$ for some unit $\theta$ and $0 \leq c \leq \min\{a,b\}$. When $c=\min\{a,b\}$, the above short exact sequence splits. If $c < \min\{a,b\}$ then we have the following non-canonical isomorphism 
    \begin{equation*}
        E \cong \{x,y| p^ax=0, p^by=p^c \theta x\} \cong  \mathbb{Z}/p^c \{ \theta x- p^{b-c}y \}\oplus \mathbb{Z}/p^{a+b-c}\{y\}.
    \end{equation*}
    Notice that two different classes in $\operatorname{Ext}^1(\mathbb{Z}/p^b, \mathbb{Z}/p^a)$ can induce isomorphic middle groups when they have the same $p$-adic valuation.
\end{remark}

\begin{prop} \label{finite2ncoskeleton}
    Let $n$ be a nonzero integer, and write $2n= p^ls$ where $(p,s)=1$. Then we have
    \begin{equation*}
        \sphere_{\K(1)}^m\left(P_{2n}^{2n+(l+1)q}\right) \cong
          \left\{
        \begin{aligned}
           & \mathbb{Z}/p^{\min\{l+1,v_p(k)+1\}},   \quad \text{if } m=2k, 2k+1 \text{ and } 2k \not\equiv  2n \mod{q};\\
        &\mathbb{Z}_p,   \quad \text{if } m=2n;\\
           & \mathbb{Z}_p \oplus \mathbb{Z}/p^l,   \quad\text{if } m=2n+1;\\
              & \mathbb{Z}/p^{\min\{l+1,v_p(rq+2n)+1\}},  \quad\text{if } m=rq+2n, r \neq0; \\
               &\mathbb{Z}/p^{l} \oplus \mathbb{Z}/p^{v_p(rq)+2},   \quad\text{if } m=rq+2n+1, r \neq 0 \text{ and } l \leq v_p(rq);\\
           & \mathbb{Z}/p^{\min\{l,v_p(rq+2n)\}+1} \oplus \mathbb{Z}/p^{v_p(rq)+1},  \quad\text{if } m=rq+2n+1, r \neq 0 \text{ and } l > v_p(rq);\\
          &0,  \quad\text{otherwise.}\\
        \end{aligned}
        \right.
    \end{equation*}
\end{prop}

\begin{proof}
  Consider the cofiber sequence 
  \begin{equation} \label{coffor3.7}
      \sphere^{2n} \rightarrow P_{2n}^{2n+(l+1)q} \rightarrow P_{2n+1}^{2n+(l+1)q}.
  \end{equation} By \cref{Xinoextension}, the calculation is reduced to determine the extension type of the following exact sequences
\begin{gather}
0 \rightarrow \mathbb{Z}/p^{\min\{l+1,v_p(rq+2n)+1\}} \rightarrow\sphere_{\K(1)}^{rq+2n+1}(P_{2n}^{2n+(l+1)q}) \rightarrow  \mathbb{Z}/p^{v_p(rq)+1}\stackrel{(\ast)}{\longrightarrow} \mathbb{Z}/p^{\min\{l+1,v_p(rq+2n)+1\}} , \quad r \neq 0; \label{first} \\
0 \rightarrow \sphere_{\K(1)}^{2n}(P_{2n}^{2n+(l+1)q}) \rightarrow \mathbb{Z}_p \stackrel{(\star)}{\longrightarrow} \mathbb{Z}/p^{l+1} \rightarrow \sphere_{\K(1)}^{2n+1}(P_{2n}^{2n+(l+1)q}) \rightarrow \mathbb{Z}_p \rightarrow 0. \label{second}
\end{gather}
 Notice that the map $(\ast)$ is the zero map because it is induced by $P_{2n+1}^{2n+(l+1)q} \rightarrow \Sigma \sphere^{2n}$, which factors through $(X_i)_{2n+1}^{2n+(l+1)q}$ for $i \equiv n \pmod{p-1}$ and $\sphere_{\K(1)}^{rq+2n+2}((X_i)_{2n+1}^{2n+(l+1)q})=0$.
 
From \cref{attachcell}, we know the cofiber sequence (\ref{coffor3.7}) maps to the cofiber sequence $\sphere^{2n} \rightarrow \Sigma^{2n}C\bar{\alpha}_{l+1} \rightarrow \sphere^{2n+(l+1)q}$. At the level of short exact sequences, it implies that the top cell projection
$\rho \colon P_{2n+1}^{2n+(l+1)q} \rightarrow \sphere^{2n+(l+1)q} $ 
determines the extension type of $(\ref{first})$ and $(\ref{second})$.

Because the map $\rho$ factors through
$P^{2n+(l+1)q}_{2n+(l+1)q-1} \cong \Sigma^{2n+(l+1)q-1} \sphere/p$, its induced map on $\sphere_{\K(1)}$-cohomology groups at cohomological degree $rq+2n+1$ has two cases: 
\begin{gather*}
   \text{If } v_p(rq+2n) < l, \text{ then } \rho^* \colon \mathbb{Z}/p^{v_p(r-l-1)+1}  \twoheadrightarrow  \mathbb{Z}/p  \stackrel{0}{\rightarrow} \mathbb{Z}/p^{v_p(rq+2n)+1}, \quad 1 \mapsto 0; \\
      \text{If } v_p(rq+2n) \geq l, \text{ then } \rho^* \colon \mathbb{Z}/p^{v_p(r-l-1)+1}  \twoheadrightarrow  \mathbb{Z}/p  \hookrightarrow \mathbb{Z}/p^{l+1},\quad  1 \mapsto p^l;
\end{gather*}
At cohomological degree $2n+1$, we have $v_p(2n)+1=l+1$ which shows the top cell projection maps $1$ to $\alpha_{l+1}$. From the commutative diagram linking the two cofiber sequences, we see the map $(\star)$ in the exact sequence (\ref{second}) has image equal to the $p$-torsion subgroup. Hence, we learn that $ \sphere_{\K(1)}^{2n}(P_{2n}^{2n+(l+1)q}) \cong \mathbb{Z}_p$ and  $ \sphere_{\K(1)}^{2n+1}(P_{2n}^{2n+(l+1)q}) \cong \mathbb{Z}_p \oplus \mathbb{Z}/p^l$.

It remains to determine the extension class of (\ref{first}). If $l>v_p(rq)$, then we have $v_p(rq+2n)=\min\{v_p(rq),l\}=v_p(rq) <l$. Hence, the top cell projection is the zero map and the short exact sequence (\ref{first}) must split. If $l \leq v_p(rq)$, then $v_p(rq+2n) \geq \min\{l, v_p(rq)\}=l$. In this situation, the top cell projection map induces the following map of $\operatorname{Ext}^1$-groups
\begin{equation*}
    \mathbb{Z}/p^{\min\{v_p(r-l-1),v_p(rq)\}+1} \rightarrow \mathbb{Z}/p^{l+1}, \quad 1 \mapsto [p^l].
\end{equation*}
By \cref{calphal}, the extension for $C\bar{\alpha}_{l+1}$ represents the unit class in $\operatorname{Ext}^1$, which implies the extension class of (\ref{first}) is $[p^l]$, and the corresponding short exact sequence gives us 
$\sphere_{K(1)}^{rq+2n+1}(P_{2n}^{2n+(l+1)q})\cong \mathbb{Z}/p^{l} \oplus \mathbb{Z}/p^{v_p(rq)+2}.$
\end{proof}
 
Finally, we are able to calculate $\sphere_{\K(1)}^*(P_{2n})$ as follows:
\begin{prop} \label{P2ncalc}
    Let $n$ be a nonzero integer and write $2n= p^ls$ where $(p,s)=1$. Then we have
    \begin{equation*}
        \sphere_{\K(1)}^m\left(P_{2n}\right) \cong
          \left\{
        \begin{aligned}
           & \mathbb{Z}/p^{v_p(a)+1},  \quad \text{if } m= 2a+1 \text{ with } a\neq  0 \text{ and } a \not\equiv  n \mod{p-1};\\
           & \mathbb{Z}_p,   \quad \text{if } m=0, 2n;\\
           & \mathbb{Z}_p,   \quad \text{if } m=1 \text{ and } n \not\equiv 0 \text{ mod }p-1;\\
           &\mathbb{Z}_p \oplus \mathbb{Z}/p^{l},  \quad \text{if } m=1 \text{ and } n \equiv 0 \text{ mod }p-1, \text{ or }  m=2n+1;\\
           & \mathbb{Z}/p^{v_p(kq+2n)+1} \oplus \mathbb{Z}/p^{v_p(kq)+1},  \quad \text{if } m=kq+2n+1 \text{ with } k \neq 0, -2n/q \text{ and } l > v_p(kq);\\
           &\mathbb{Z}/p^l \oplus \mathbb{Z}/p^{v_p(kq)+v_p(kq+2n)+2-v_p(2n)},  \quad \text{if } m=kq+2n+1 \text{ with }  k \neq 0, -2n/q \text{ and } l \leq v_p(kq);\\
          &0,  \quad\text{otherwise.}\\
        \end{aligned}
        \right.
    \end{equation*}
\end{prop}

\begin{proof}
    As discussed at the beginning of the section, we want to determine the extension type of the following exact sequences:
\begin{gather}
0 \rightarrow \mathbb{Z}_p \rightarrow \sphere_{\K(1)}^0(P_{2n}) \rightarrow  0 \rightarrow \mathbb{Z}_p \rightarrow \sphere_{\K(1)}^1(P_{2n}) \rightarrow  \sphere_{\K(1)}^1(\sphere^{2n}) \rightarrow 0; \label{first2}\\
0 \rightarrow \sphere_{\K(1)}^{2n}(P_{2n}) \rightarrow \mathbb{Z}_p \stackrel{(\ast)}{\longrightarrow} \mathbb{Z}/p^{l+1} \rightarrow \sphere_{\K(1)}^{2n+1}(P_{2n}) \rightarrow \mathbb{Z}_p \rightarrow 0;\label{second2}\\
0 \rightarrow \mathbb{Z}/p^{v_p(kq+2n)+1}\rightarrow \sphere_{\K(1)}^{kq+2n+1}(P_{2n}) \rightarrow \mathbb{Z}/p^{v_p(kq)+1} \rightarrow 0, \quad k \in \mathbb{Z} \setminus \{0,-2n/q\}.\label{third2}
\end{gather}  

For the sequence (\ref{second2}), it admits a map from the exact sequence $(\ref{second})$ in the proof of \cref{finite2ncoskeleton} where every map is an isomorphism, so we learn that $ \sphere_{\K(1)}^{2n}(P_{2n}) \cong \mathbb{Z}_p$ and  $ \sphere_{\K(1)}^{2n+1}(P_{2n}) \cong \mathbb{Z}_p \oplus \mathbb{Z}/p^l$.

For the sequence (\ref{first2}), we have $\sphere_{\K(1)}^1(\sphere^{2n}) =0$ when $ n \not\equiv 0 \text{ mod }p-1 $, so the group and the adjacent maps are determined. If $ n \equiv 0 \text{ mod }p-1 $, we have $\sphere_{\K(1)}^1(\sphere^{2n}) \cong \mathbb{Z}/p^{l+1}$, and since it maps to $0\rightarrow \mathbb{Z}/p^{l+1} \rightarrow \mathbb{Z}/p^l \oplus \mathbb{Z}/p^{l+2} \rightarrow \mathbb{Z}/p^{l+1} \rightarrow 0$ by \cref{attachcell}, this implies that $ \sphere_{\K(1)}^1( P_{2n}) \cong \mathbb{Z}_p \oplus \mathbb{Z}/p^{l}, \quad \text{if } n \equiv 0\mod{p-1}.$

For the sequence (\ref{third2}), \cref{attachcell} implies there is a map
\begin{equation*}
    \operatorname{Ext}^1(\mathbb{Z}/p^{v_p(kq)+1},\mathbb{Z}/p^{v_p(kq+2n)+1})  \rightarrow \operatorname{Ext}^1(\mathbb{Z}/p^{v_p(kq)+1},\mathbb{Z}/p^{\min\{v_p(kq+2n)+1,l+1\}}) 
\end{equation*}
induced by the natural surjection. The undetermined extension type $[\gamma]$ maps to the extension type of $\sphere_{\K(1)}^{kq+2n+1}(P_{2n}^{2n+(l+1)q})$. The map of $\operatorname{Ext}^1$ groups is the natural projection
$$\mathbb{Z}/p^{\min\{v_p(kq+2n), v_p(kq)\}+1}  \doublearrow{} \mathbb{Z}/p^{\min\{l, v_p(kq)\}+1}$$
with $[\gamma]$ mapping to $[p^l]$. 

If $l > v_p(kq) $, we have $[p^l]=0$ and $v_p(kq+2n)=v_p(kq)$, so the above map becomes $\mathbb{Z}/p^{v_p(kq)+1} \stackrel{\operatorname{id}}{\longrightarrow}\mathbb{Z}/p^{v_p(kq)+1}$. Hence we have $[\gamma]=0$.

If $l \leq v_p(kq) $, we see the extension type $[\gamma]$ is in the pre-image of $[p^{l}]$ as
$$\{p^{l} + p^{l+1}a \mid 0 \leq a < p^{\min\{v_p(kq+2n), v_p(kq)\}-l} \}.$$

Although the extension type is not uniquely determined, the different extension types in the pre-image give isomorphic middle groups as
    $\mathbb{Z}/p^{l} \oplus  \mathbb{Z}/p^{v_p(kq)+2+v_p(kq+2n)-l}$, and this completes the computation.
\end{proof}

The full description of the additive structure of $\pi_{i,j} b(\sphere_{\K(1)})$ is presented as follows:

\begin{theorem} \label{Main result}
The  $\ROS(C_p)$-graded homotopy groups of $b(\KS)$, as abelian groups, are as follows: 
\begin{center}
\begin{tabular}{ |c|c|c|c|} 
\hline
 $\pi_{i,j}b(\sphere_{\K(1)}), \quad j=2k+1$ & $i=2k, 2k+1$ & $2m, m\neq k$ & otherwise \\
\hline
 & $\mathbb{Z}_p$ & $\mathbb{Z}/p^{v_p(k-m)+1}$ & $0$\\ 
\hline
\end{tabular}

\begin{tabular}{ |c|c|c|c|} 
\hline
 $\pi_{i,j}b(\sphere_{\K(1)}), \quad j=0$ & $i=-1,0$ & $i=2m-1$ & otherwise \\
\hline
&$\mathbb{Z}_p \oplus \mathbb{Z}_p$ && 0\\
\hline
$m \neq 0$ and $m \not\equiv 0$ mod $p-1$ & & $\mathbb{Z}/p^{v_p(m)+1}$ &\\
\hline
$m\neq 0$ and $m\equiv 0$ mod $p-1$ && $ \mathbb{Z}/p^{v_p(m)+1} \oplus \mathbb{Z}/p^{v_p(m)+1}$&\\
\hline
\end{tabular}

\begin{tabular}{ | m{10em} | m{2cm}| m{6cm} | m{2cm}| } 
\hline
 $\pi_{i,j}b(\sphere_{\K(1)}), \quad j=2k$ with $ k\neq 0$ & $i=0,2k$ & $i=2m-1$ & otherwise \\
\hline
 & $\mathbb{Z}_p$ & &0\\ 
 \hline
$m \neq 0, k$ and $m \not\equiv 0$ mod $p-1$ & & $\mathbb{Z}/p^{v_p(k-m)+1}$ &\\
\hline
$m=0$ &&$\mathbb{Z}_p \oplus \mathbb{Z}/p^{v_p(k)}$&\\
\hline
$m=k$ with $k \not\equiv 0$ mod $p-1$&& $\mathbb{Z}_p$&\\
\hline
$m=k$ with $k \equiv 0$ mod $p-1$&& $\mathbb{Z}_p \oplus \mathbb{Z}/p^{v_p(k)}$&\\
\hline
$m \neq 0,k$ and $m\equiv 0$ mod $p-1$ with  $v_p(k) > v_p(m)$ && $ \mathbb{Z}/p^{v_p(k-m)+1} \oplus \mathbb{Z}/p^{v_p(m)+1}$&\\
\hline
$m \neq 0,k$ and $m\equiv 0$ mod $p-1$ with  $v_p(k) \leq v_p(m)$ && $ \mathbb{Z}/p^{v_p(k)} \oplus \mathbb{Z}/p^{v_p(m)+v_p(k-m)+2-v_p(k)}$&\\
\hline
\end{tabular}
\end{center}
\end{theorem}

\cref{fig:p3-spoke-chart} displays the calculation at $p=3$:

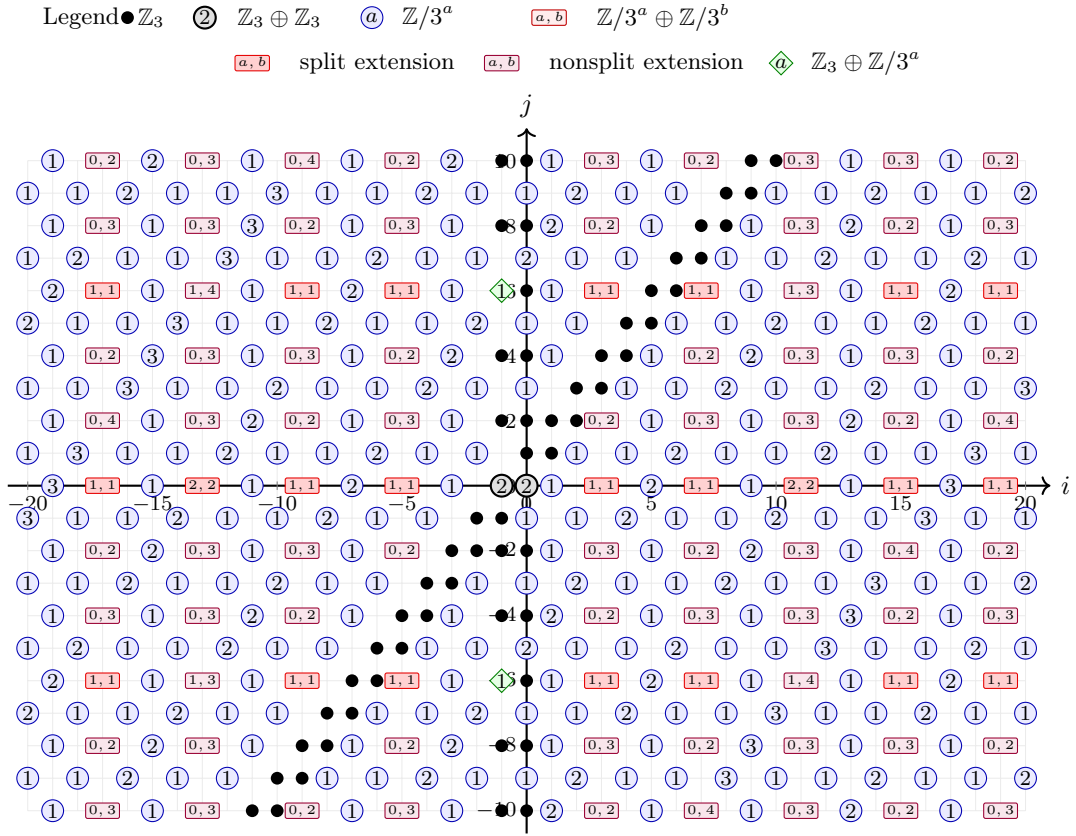
\begin{figure}[H] 
  \centering
  \newcommand{\Z}{\mathbb{Z}}
\newcommand{\Zthree}{\Z_3}

% Display window
\def\xmin{-20}
\def\xmax{20}
\def\ymin{-10}
\def\ymax{10}

\makeatletter
% 3-adic valuation of a nonzero integer; v_3(0) is set to 99 as a dummy.
\newcount\vt@n
\newcount\vt@v
\newcommand{\vthree}[2]{%
  \pgfmathtruncatemacro{\vtarg}{#1}%
  \vt@n=\vtarg\relax
  \ifnum\vt@n<0 \multiply\vt@n by -1 \fi
  \vt@v=0
  \ifnum\vt@n=0
    \vt@v=99
  \else
    \loop\ifnum\vt@n>0
      \pgfmathtruncatemacro{\vtr}{mod(\the\vt@n,3)}%
      \ifnum\vtr=0
        \divide\vt@n by 3
        \advance\vt@v by 1
      \else
        \vt@n=0
      \fi
    \repeat
  \fi
  \xdef#2{\the\vt@v}%
}
\makeatother

\tikzset{
  zthree/.style={circle,fill=black,minimum size=4.8pt,inner sep=0pt},
  zthreethree/.style={circle,draw=black,thick,fill=black!12,minimum size=8.2pt,inner sep=0pt},
  tor/.style={circle,draw=blue!70!black,fill=blue!8,minimum size=8.2pt,inner sep=0pt},
  tortwo/.style={rectangle,draw=red!75!black,fill=red!7,rounded corners=.7pt,minimum height=5.8pt,inner xsep=1.2pt,inner ysep=0pt,font=\fontsize{5}{5}\selectfont},
  tortwoTriv/.style={rectangle,draw=red!90!black,fill=red!18,rounded corners=.7pt,minimum height=5.8pt,inner xsep=1.2pt,inner ysep=0pt,font=\fontsize{5}{5}\selectfont},
  tortwoNontriv/.style={rectangle,draw=purple!80!black,fill=purple!10,rounded corners=.7pt,minimum height=5.8pt,inner xsep=1.2pt,inner ysep=0pt,font=\fontsize{5}{5}\selectfont},
  zthreetor/.style={diamond,draw=green!50!black,fill=green!10,minimum size=9pt,inner sep=0pt},
  glabel/.style={font=\scriptsize,inner sep=1pt},
  axislabel/.style={font=\small},
  ticklabel/.style={font=\scriptsize},
  legendtext/.style={font=\footnotesize}
}

\newcommand{\putZthree}[2]{\node[zthree] at (#1,#2) {};}
\newcommand{\putZthreeTwo}[2]{\node[zthreethree] at (#1,#2) {}; \node[glabel] at (#1,#2) {$2$};}
\newcommand{\putTor}[3]{\node[tor] at (#1,#2) {}; \node[glabel] at (#1,#2) {$#3$};}
\newcommand{\putTorTwo}[4]{\node[tortwo] at (#1,#2) {$#3,#4$};}
\newcommand{\putTorTwoTriv}[4]{\node[tortwoTriv] at (#1,#2) {$#3,#4$};}
\newcommand{\putTorTwoNontriv}[4]{\node[tortwoNontriv] at (#1,#2) {$#3,#4$};}
\newcommand{\putZthreeTor}[3]{\node[zthreetor] at (#1,#2) {}; \node[glabel] at (#1,#2) {$#3$};}

\begin{tikzpicture}[x=.33cm,y=.43cm]
  % grid and axes
  \draw[step=1,gray!16,very thin] (\xmin,\ymin) grid (\xmax,\ymax);
  \draw[->,thick] (\xmin-.8,0) -- (\xmax+1.0,0) node[right,axislabel] {$i$};
  \draw[->,thick] (0,\ymin-.7) -- (0,\ymax+1.0) node[above,axislabel] {$j$};
  \foreach \x in {-20,-15,-10,-5,0,5,10,15,20} {
    \draw[black!50] (\x,0.13) -- (\x,-0.13);
    \node[below,ticklabel] at (\x,0) {$\x$};
  }
  \foreach \y in {-10,-8,-6,-4,-2,0,2,4,6,8,10} {
    \draw[black!50] (0.13,\y) -- (-0.13,\y);
    \node[left,ticklabel] at (0,\y) {$\y$};
  }

  % Main plotting loop, specialized to p=3, so p-1=2 and q=4.
  % This version follows the updated Theorem 3.9 in a_newer_version (4).pdf.
  \foreach \J in {-10,-9,...,10}{
    \foreach \I in {-20,-19,...,20}{
      \ifodd\J
        % Odd row: j=2k+1.
        % Updated rule: i=2k and i=2k+1 give Z_3;
        % all other even i=2m with m != k give Z/3^{v_3(k-m)+1}.
        \pgfmathtruncatemacro{\K}{(\J-1)/2}
        \pgfmathtruncatemacro{\TwoK}{2*\K}
        \pgfmathtruncatemacro{\TwoKPlusOne}{2*\K+1}
        \ifnum\I=\TwoK
          \putZthree{\I}{\J}%
        \else\ifnum\I=\TwoKPlusOne
          \putZthree{\I}{\J}%
        \else
          \ifodd\I
          \else
            \pgfmathtruncatemacro{\M}{\I/2}
            \ifnum\M=\K
            \else
              \vthree{\K-\M}{\A}
              \pgfmathtruncatemacro{\E}{\A+1}
              \putTor{\I}{\J}{\E}% represents Z/3^E
            \fi
          \fi
        \fi\fi
      \else
        \ifnum\J=0
          % Middle row: j=0.
          % Updated rule: i=-1 and i=0 give Z_3 \oplus Z_3.
          \ifnum\I=-1
            \putZthreeTwo{\I}{\J}%
          \else\ifnum\I=0
            \putZthreeTwo{\I}{\J}%
          \else
            \ifodd\I
              \pgfmathtruncatemacro{\M}{(\I+1)/2}
              \ifnum\M=0\else
                \vthree{\M}{\A}
                \pgfmathtruncatemacro{\E}{\A+1}
                \ifodd\M
                  \putTor{\I}{\J}{\E}%
                \else
                  % Middle row j=0: this direct-sum case is a split/trivial extension.
                  \putTorTwoTriv{\I}{\J}{\E}{\E}%
                \fi
              \fi
            \fi
          \fi\fi
        \else
          % Even row: j=2k, k not zero.
          \pgfmathtruncatemacro{\K}{\J/2}
          \pgfmathtruncatemacro{\TwoK}{2*\K}
          \ifnum\I=0
            \putZthree{\I}{\J}%
          \else\ifnum\I=\TwoK
            \putZthree{\I}{\J}%
          \else
            \ifodd\I
              \pgfmathtruncatemacro{\M}{(\I+1)/2}
              \ifnum\M=0
                \vthree{\K}{\C}
                \ifnum\C=0 \putZthree{\I}{\J}\else \putZthreeTor{\I}{\J}{\C}\fi
              \else\ifnum\M=\K
                \vthree{\K}{\C}
                \ifodd\K
                  \putZthree{\I}{\J}%
                \else
                  \ifnum\C=0 \putZthree{\I}{\J}\else \putZthreeTor{\I}{\J}{\C}\fi
                \fi
              \else
                \ifodd\M
                  \vthree{\K-\M}{\A}
                  \pgfmathtruncatemacro{\E}{\A+1}
                  \putTor{\I}{\J}{\E}%
                \else
                  \vthree{\K-\M}{\A}% A = v_3(k-m)
                  \vthree{\M}{\B}% B = v_3(m)
                  \vthree{\K}{\C}% C = v_3(k)
                  % Updated Theorem 3.9 in a_newer_version (4):
                  % split if v3(k)>v3(m), with exponents v3(k-m)+1 and v3(m)+1;
                  % nonsplit if v3(k)<=v3(m), with exponents v3(k) and
                  % v3(m)+v3(k-m)+2-v3(k).
                  \ifnum\C>\B
                    \pgfmathtruncatemacro{\Eone}{\A+1}
                    \pgfmathtruncatemacro{\Etwo}{\B+1}
                    \putTorTwoTriv{\I}{\J}{\Eone}{\Etwo}%
                  \else
                    \pgfmathtruncatemacro{\Eone}{\C}
                    \pgfmathtruncatemacro{\Etwo}{\B+\A+2-\C}
                    \putTorTwoNontriv{\I}{\J}{\Eone}{\Etwo}%
                  \fi
                \fi
              \fi\fi
            \fi
          \fi\fi
        \fi
      \fi
    }
  }

  % legend, moved above the picture with wider spacing
  \node[anchor=west,legendtext] at (-19.8,14.4) {Legend};

  % first row
  \node[zthree] at (-16,14.4) {};
    \node[anchor=west,legendtext] at (-16.0,14.4) {$\Z_3$};

  \node[zthreethree] at (-12.9,14.4) {};
    \node[glabel] at (-12.9,14.4) {$2$};
    \node[anchor=west,legendtext] at (-12.0,14.4) {$\Z_3\oplus\Z_3$};

  \node[tor] at (-6.2,14.4) {};
    \node[glabel] at (-6.2,14.4) {$a$};
    \node[anchor=west,legendtext] at (-5.4,14.4) {$\Z/3^a$};

  \node[tortwo] at (0.9,14.4) {$a,b$};
    \node[anchor=west,legendtext] at (2.4,14.4) {$\Z/3^a\oplus\Z/3^b$};

  % second row
  \node[tortwoTriv] at (-11.0,13.0) {$a,b$};
    \node[anchor=west,legendtext] at (-9.5,13.0) {split extension};

  \node[tortwoNontriv] at (-1.0,13.0) {$a,b$};
    \node[anchor=west,legendtext] at (0.5,13.0) {nonsplit extension};

  \node[zthreetor] at (10.2,13.0) {};
    \node[glabel] at (10.2,13.0) {$a$};
    \node[anchor=west,legendtext] at (11.0,13.0) {$\Z_3\oplus\Z/3^a$};
\end{tikzpicture}
  \caption{The spoke-graded homotopy groups of $b(S_{\K(1)})$ for $p=3$.}
  \label{fig:p3-spoke-chart}
\end{figure}

\section{Equivariant structure of the spoke homotopy groups}\label{equivariantstructure}
 Let $X$ be an arbitrary $C_p$-equivariant spectrum. We use $C_{\sigma}$ to denote the underlying spectrum of $\sphere^{\sigma}$ equipped with a residual $C_p$-action.  The spoke-graded groups carry three basic families of structure maps:
 \begin{itemize}
     \item The $a_\sigma$-map $\pi_{i,j} X\rightarrow \pi_{i-1,j-1}X$ induced by $\sphere \rightarrow \sphere^{\sigma}$.
     \item The restriction maps 
     $$\res: \pi_{i,2k}X \rightarrow \pi_i^e X \text{ and } \pi_{i,2k+1}X \rightarrow \pi_{i-1} \Map(C_{\sigma}, X^e) \cong \oplus_{p-1} \pi_i X, $$
   which arise naturally from the Mackey functor structure.     \item The transfer maps
     $$
     \tr: \pi_i^eX\rightarrow \pi_{i,2k} X
     \text{ and } \pi_{i-1} \Map(C_{\sigma}, X^e) \rightarrow \pi_{i,2k+1} X,
     $$
     which also arise from the Mackey functor structure.
     \end{itemize}
 The structure maps are connected to each other by the following relations:
 \begin{itemize}
     \item For $j=2k$, we have $\operatorname{Ker}(\res\colon \pi_{i,2k} X \rightarrow \pi_{i} X^e) = \operatorname{Im}(a_{\sigma} \colon \pi_{i+1, 2k+1}X \rightarrow \pi_{i,2k} X).$
     \item For $j=2k+1$, we have 
     $$\operatorname{Ker}(\res\colon \pi_{i,2k+1} X \rightarrow \pi_{i-1} \Map(C_{\sigma}, X^e)) = \operatorname{Im}( \pi_{i-1, 2k}F(\sphere^{\sigma} \wedge \sphere^{\sigma}, X^e) \rightarrow \pi_{i,2k+1} X),$$
     which contains $\operatorname{Im}(a_{\sigma}: \pi_{i+1,2k+2} X \rightarrow \pi_{i,2k+1} X)$ as a subgroup.
     \item For $j=2k$, we have $\operatorname{Im} (\tr \colon \pi_i X^e \rightarrow \pi_{i, 2k}X) = \operatorname{Ker}(a_{\sigma} \colon \pi_{i,2k} X \rightarrow \pi_{i-1, 2k-1} X).$
     \item For $j=2k+1$, we have 
     $$
     \operatorname{Im} (\tr \colon \pi_{i-1} \Map(C_{\sigma}, X^e) \rightarrow \pi_{i,2k+1}X) = \operatorname{Ker}(\pi_{i, 2k+1} X \rightarrow \pi_{i-1,2k} F(\sphere^{\sigma} \wedge \sphere^{-\sigma}, X) )
     $$
     which is a subgroup of $\operatorname{Ker}(a_{\sigma} \colon \pi_{i,2k+1}  X \rightarrow \pi_{i-1,2k} X)$.
     \item The composites of restriction and transfer satisfy the relations
     \begin{equation*}
         \res \circ \tr = \sum_{g \in C_p} g, \quad \tr \circ \res =[C_p/e] \cdot (-),
     \end{equation*}
     where $[C_p/e]$ means the action of $[C_p/e]$ on $\pi_{*,*}X$ as a $\pi_{0,0} \sphere \cong A(C_p)$-module.   
 \end{itemize}
We will describe these structure maps up to multiplication by a unit in this section.
\subsection{The $a_{\sigma}$-map} In our case, the $\mathbb{Z}[a_{\sigma}]$-module structure is described by the map $$\pi_{i,j}b(\sphere_{\K(1)})\cong \sphere_{\K(1)}^{j-i}(P_j)\rightarrow \pi_{i-1,j-1} b(\sphere_{\K(1)}) \cong \sphere_{\K(1)}^{j-i}(P_{j-1})$$
induced by the natural map $P_{j-1} \rightarrow P_j$. When $j=2k+1$, this map is part of the cofiber sequence $\sphere^{2k} \rightarrow P_{2k} \rightarrow P_{2k+1}.$ When $k=0$, this cofiber sequence splits and the $a_{\sigma}$-map is an injection. When $k \neq 0$, the maps can be read off directly from the calculation in \cref{sec: evencom} : 
\begin{center}
\begin{tabular}{ | m{8cm} | m{10cm}| } 
\hline
 $\mathbf{j=2k+1}$ \bf{with} $ \mathbf{k\neq 0}$ & \bf{The $\mathbf{a_{\sigma}}$-map} \\
\hline
$i=2k+1$  & $\mathbb{Z}_p \stackrel{\cong}{\rightarrow}\mathbb{Z}_p$   \\ 
 \hline
 $i=1$  & $0 \rightarrow \mathbb{Z}_p$   \\ 
 \hline
 $i=0$ &$\mathbb{Z}/p^{v_p(k)+1} \stackrel{(0,1)}{\rightarrow}\mathbb{Z}_p \oplus \mathbb{Z}/p^{v_p(k)}$\\
\hline
$i=2k$ with $k \not\equiv 0$ mod $p-1$& $\mathbb{Z}_p \stackrel{\cong}{\rightarrow}\mathbb{Z}_p$\\
\hline
$i=2k$ with $k \equiv 0$ mod $p-1$& $\mathbb{Z}_p \stackrel{(p,1)}{\longrightarrow}\mathbb{Z}_p \oplus \mathbb{Z}/p^{v_p(k)}$\\
\hline
$i=2m, m \neq 0, k$ and $m \not\equiv 0$ mod $p-1$  & $\mathbb{Z}/p^{v_p(k-m)+1}\stackrel{\cong}{\rightarrow}\mathbb{Z}/p^{v_p(k-m)+1}$ \\
\hline
$i=2m, m \neq 0,k$ and $m\equiv 0$ mod $p-1$ with  $v_p(k) > v_p(m)$ & $ \mathbb{Z}/p^{v_p(k-m)+1}\stackrel{(1,0)}{\longrightarrow} \mathbb{Z}/p^{v_p(k-m)+1} \oplus \mathbb{Z}/p^{v_p(m)+1}$\\
\hline
$i=2m, m \neq 0,k$ and $m\equiv 0$ mod $p-1$ with  $v_p(k) \leq v_p(m)$ & $ \mathbb{Z}/p^{v_p(k-m)+1}\xrightarrow{(1,p^{v_p(m)-v_p(k)+1})}\mathbb{Z}/p^{v_p(k)} \oplus \mathbb{Z}/p^{v_p(m)+v_p(k-m)+2-v_p(k)}$\\
\hline
\end{tabular}
\end{center}
When $j=2k$, the $a_{\sigma}$ map is restricted by the following information:
\begin{itemize}
    \item For the composite map $P_{2k-1} \to P_{2k} \rightarrow P_{2k+1}$, it induces the multiplication by $p$ map on the $\sphere_{\K(1)}$-cohomology groups at degree $rq+2k+1$ because its fiber is $\Sigma^{2k-1}\sphere/p$.
    \item The long exact sequence induced by the cofiber sequence $P_{2k-1} \rightarrow P_{2k} \rightarrow \sphere^{2k}$.
\end{itemize}
With these observations, the two exact sequences determine the entries in the following tables:
\begin{center}
\begin{tabular}{ | m{8cm} | m{10cm}| } 
\hline
 $\mathbf{j=0}$ & \bf{The $\mathbf{a_{\sigma}}$-map} \\
\hline
 $i=0,-1$ &$\mathbb{Z}_p \oplus \mathbb{Z}_p \xrightarrow{(1,p)}  \mathbb{Z}_p$\\
\hline
$i=2m-1$,  $m \not\equiv 0$ mod $p-1$  & $\mathbb{Z}/p^{v_p(m)+1}\stackrel{\cong}{\rightarrow}\mathbb{Z}/p^{v_p(m)+1}$ \\
\hline
$i=2m-1$ and $m\equiv 0$ mod $p-1$ & $  \mathbb{Z}/p^{v_p(m)+1} \oplus \mathbb{Z}/p^{v_p(m)+1} \xrightarrow{(1,p)}\mathbb{Z}/p^{v_p(m)+1}$\\
\hline
\hline
\end{tabular}
\end{center}

\begin{center}
\begin{tabular}{ | m{8cm} | m{10cm}| } 
\hline
 $\mathbf{j=2k}$ \bf{with} $ \mathbf{k\neq 0}$ & \bf{The $\mathbf{a_{\sigma}}$-map} \\
\hline
$i=2k$ with $k \not\equiv 0$ mod $p-1$& $\mathbb{Z}_p \stackrel{\cong}{\rightarrow}\mathbb{Z}_p$\\
\hline
$i=2k$ with $k \equiv 0$ mod $p-1$ & $\mathbb{Z}_p \stackrel{p}{\rightarrow}\mathbb{Z}_p$\\
\hline
 $i=0$  & $\mathbb{Z}_p \rightarrow 0$   \\ 
\hline
 $i=-1$ &$\mathbb{Z}_p \oplus \mathbb{Z}/p^{v_p(k)} \xrightarrow{(1,p)}  \mathbb{Z}/p^{v_p(k)+1}$\\
\hline
$i=2k-1$ with $k \not\equiv 0$ mod $p-1$& $\mathbb{Z}_p \stackrel{\cong}{\rightarrow}\mathbb{Z}_p$\\
\hline
$i=2k-1$ with $k \equiv 0$ mod $p-1$& $\mathbb{Z}_p \oplus \mathbb{Z}/p^{v_p(k)} \xrightarrow{(1,0)} \mathbb{Z}_p$\\
\hline
$i=2m-1, m \neq 0, k$ and $m \not\equiv 0$ mod $p-1$  & $\mathbb{Z}/p^{v_p(k-m)+1}\stackrel{\cong}{\rightarrow}\mathbb{Z}/p^{v_p(k-m)+1}$ \\
\hline
$i=2m-1, m \neq 0,k$ and $m\equiv 0$ mod $p-1$ with  $v_p(k) > v_p(m)$ & $  \mathbb{Z}/p^{v_p(k-m)+1} \oplus \mathbb{Z}/p^{v_p(m)+1} \xrightarrow{(p,1)}\mathbb{Z}/p^{v_p(k-m)+1}$\\
\hline
$i=2m-1, m \neq 0,k$ and $m\equiv 0$ mod $p-1$ with  $v_p(k) \leq v_p(m)$ & $ \mathbb{Z}/p^{v_p(k)} \oplus \mathbb{Z}/p^{v_p(m)+v_p(k-m)+2-v_p(k)}\xrightarrow{(p^{v_p(k-m)+1-v_p(k)},1)}\mathbb{Z}/p^{v_p(k-m)+1}$\\
\hline
\end{tabular}
\end{center}
All cases follow from the same kernel-and-cokernel calculation. We give the final case, where both summands contribute, since it is the only case not immediate from cyclicity. Suppose the map can be written as $(A,B)$. Because the map is surjective and $v_p(k-m)+1 >v_p(k)$, we see $B$ must be a unit and we can choose it to be $1$. Then $(1,-A)$ generates the kernel of the map, and has order $p^{v_p(m)+1}$ from the short exact sequence. This implies that
\begin{equation*}
    v_p(A)= v_p(k-m)+v_p(m)+2 -v_p(k)-v_p(m)-1 =v_p(k-m)-v_p(k)+1.
\end{equation*}

\subsection{The restriction maps} We next determine the restriction maps. We divide the computation into two cases. When $j=2k$, the restriction map is simply
\begin{equation*}
    \res \colon \pi_{i,2k} b(\sphere_{\K(1)}) \cong \sphere_{\K(1)}^{2k-i} (P_{2k}) \rightarrow \pi_i(\sphere_{K(1)}) \cong \sphere_{K(1)}^{2k-i}(\sphere^{2k}),
\end{equation*}
which is induced by the bottom inclusion map $\sphere^{2k} \rightarrow P_{2k}$. This was already calculated in the previous section. When $k=0$, it is the natural surjection. For $k \neq 0$, we have the following table
\begin{center}
\begin{tabular}{ | m{8cm} | m{10cm}| } 
\hline
 $\mathbf{j=2k}$ \bf{with} $ \mathbf{k\neq 0}$ & \bf{The restriction map} \\
 \hline
 $i=0$  & $\mathbb{Z}_p \xrightarrow{p} \mathbb{Z}_p$   \\ 
 \hline
 $i=-1$ &$\ \mathbb{Z}_p \oplus \mathbb{Z}/p^{v_p(k)} \xrightarrow[]{(1,0)} \mathbb{Z}_p$ \\
\hline
$i=2k-1$ with $k \not\equiv 0$ mod $p-1$& $\mathbb{Z}_p \rightarrow  0$\\
\hline
$i=2k-1$ with $k \equiv 0$ mod $p-1$& $\mathbb{Z}_p \oplus \mathbb{Z}/p^{v_p(k)} \xrightarrow[]{(1,-p)} \mathbb{Z}/p^{v_p(k)+1}$\\
\hline
$i=2m-1, m \neq 0,k$ and $m\equiv 0$ mod $p-1$ with  $v_p(k) > v_p(m)$ & $  \mathbb{Z}/p^{v_p(k-m)+1} \oplus \mathbb{Z}/p^{v_p(m)+1} \xrightarrow[]{(0,-1)} \mathbb{Z}/p^{v_p(m)+1}$\\
\hline
$i=2m-1, m \neq 0,k$ and $m\equiv 0$ mod $p-1$ with  $v_p(k) \leq v_p(m)$ & $ \mathbb{Z}/p^{v_p(k)} \oplus \mathbb{Z}/p^{v_p(m)+v_p(k-m)+2-v_p(k)} \xrightarrow[]{(-p^{v_p(m)-v_p(k)+1},1)} \mathbb{Z}/p^{v_p(m)+1}$\\
\hline
For all other $i$ & $\pi_{i,j}b(\sphere_{K(1)}) \rightarrow 0$\\
\hline
\end{tabular}
\end{center}
For the odd tails $j=2k+1$,  we have the following proposition:
\begin{prop}\label{resodd}
 For any $i,k \in \mathbb{Z}$, the restriction map
 \begin{equation*}
    \res \colon \pi_{i,2k+1} b(\sphere_{K(1)}) \rightarrow \pi_{i-1} \Map(C_{\sigma}, \sphere_{K(1)})
 \end{equation*}
 is the zero map except at degrees $(rq-1, rq-1)$ for some $r \neq0$. In this case, we have the restriction map
    \begin{equation*}
              \res \colon 
        \mathbb{Z}_p \twoheadrightarrow \mathbb{Z}/p  \xhookrightarrow{[(0, \alpha_l, 2\alpha_l, \cdots, (p-1)\alpha_l)]} \mathbb{Z}/p^{v_p(r)+1} \otimes \mathcal{Q}.
    \end{equation*}
\end{prop}

\begin{proof}
Recall that there is a residual $C_p$-action on $\pi_{i-1}\Map(C_{\sigma}, \sphere_{\K(1)})$, and the restriction map factors through the fixed points of this action. Consider the cofiber sequence $ C_{\sigma} \rightarrow \sphere^1 \wedge {C_p}_+ \rightarrow \sphere^1$, applying $F(-, \sphere_{\K(1)})$ yields a $C_p$-equivariant long exact sequence that breaks as
\begin{equation*}
    0 \rightarrow \pi_i \sphere_{\K(1)} \xrightarrow[]{\Delta} \pi_i\sphere_{\K(1)} \otimes C_p \rightarrow \pi_{i-1}\Map(C_{\sigma},\sphere_{\K(1)}) \rightarrow 0.
\end{equation*}
Let $\mathcal{Q}$ be a $C_p$-module $ \mathbb{ Z}[C_p]/(1+\sigma + \cdots + \sigma^{p-1})$. We identify $\pi_{i-1}\Map(C_{\sigma},\sphere_{\K(1)})$ as $\pi_i \sphere_{K(1)} \otimes \mathcal{Q}$, and the $C_p$-action is of the form
\begin{equation*}
    \sigma [(a_0, \cdots, a_{p-1})]=[(a_1, \cdots, a_{p-1}, a_0)].
\end{equation*}
Its $C_p$-invariants can be identified via the isomorphism
\begin{equation*}
    \pi_i(\sphere_{\K(1)})[p] \xrightarrow[]{\cong} (\pi_{i-1}\Map(C_{\sigma},\sphere_{\K(1)}))^{C_p}, \quad \alpha_l \mapsto [(0, \alpha_l, 2\alpha_l, \cdots, (p-1)\alpha_l)].
\end{equation*}
For degree reasons, the restriction maps can be non-zero only in the following cases:
\begin{itemize}
    \item When $k=0$ and $i=0$, $\pi_{0,1} b(\sphere_{\K(1)}) \rightarrow  \pi_{-1} \Map(C_{\sigma}, \sphere_{\K(1)}), \quad \mathbb{Z}_p \rightarrow  \mathbb{Z}_p \otimes \mathcal{Q}$;
    \item When $k \neq 0$ and $i=0$, $\pi_{0,2k+1} b(\sphere_{\K(1)}) \rightarrow \pi_{-1} \Map(C_{\sigma}, \sphere_{\K(1)}), \quad \mathbb{Z}/p^{v_p(k)+1} \rightarrow\mathbb{Z}_p \otimes \mathcal{Q} $;
    \item  When $2k+1 = mq-1$ for some integer $m$ and $i = 2k+1$, 
    $$\pi_{2k+1,2k+1} b(\sphere_{\K(1)}) \rightarrow \pi_{2k} \Map(C_{\sigma}, \sphere_{\K(1)}), \quad \mathbb{Z}_p \rightarrow \mathbb{Z}/p^{v_p(m)+1} \otimes \mathcal{Q} .$$
\end{itemize}
The first two cases are zero maps because the targets have no $p$-torsion elements. For the last case, when $m\neq 0$, the target also has no $p$-torsion. When $m\neq 0$, we use the relation
$$\operatorname{Ker}(\res\colon \pi_{i,2k+1} X \rightarrow \pi_{i-1} \Map(C_{\sigma}, X^e) )\cong \operatorname{Im}( \pi_{i-1, 2k}F(\sphere^{\sigma} \wedge \sphere^{\sigma}, X) \rightarrow \pi_{i,2k+1} X).$$
By \cref{Ssigma}, we have $\sphere^{\sigma} \wedge  \sphere^{\sigma} \cong \sphere^{\lambda} \oplus \bigvee_{p-2} \Sigma^2 {C_p}_+$, which implies
\begin{equation*}
    \pi_{i-1, 2k}F(\sphere^{\sigma} \wedge \sphere^{\sigma}, b(\sphere_{\K(1)})) \cong \pi_{i+1,j+1} b(\sphere_{\K(1)}) \oplus (\pi_{i+1}\sphere_{\K(1)})^{p-2}.
\end{equation*}
Notice that $\pi_{i+1}\sphere_{K(1)} = \pi_{mq} \sphere_{K(1)} =0$ because $m\neq 0$. We see
\begin{equation*}
    \operatorname{Ker}(\res) \cong \operatorname{Im} (a_{\sigma} \colon \pi_{mq, mq} b(\sphere_{K(1)}) \rightarrow \pi_{mq-1, mq-1} b(\sphere_{K(1))})).
\end{equation*}
Our computation of $a_{\sigma}$ shows that this map is the times $p$ map with image being $p\mathbb{Z}$, which implies that the restriction map 
\end{proof}

\subsection{The transfer maps} Finally, we determine the transfer maps. For the even tails, the transfer map is
\begin{equation*}
   \tr \colon \pi_{i} \sphere_{K(1)} \rightarrow \pi_{i,2k} b(\sphere_{\K(1)}).
\end{equation*}
Using the relation that the image of the transfer equals the kernel of $a_{\sigma}$, we have the following tables:
\begin{center}
\begin{tabular}{ | m{8cm} | m{10cm}| } 
\hline
 $\mathbf{j=0}$ & \bf{The transfer map} \\
\hline
 $i=0,-1$ &$\mathbb{Z}_p  \xrightarrow{(p,-1)}  \mathbb{Z}_p \oplus \mathbb{Z}_p$\\
\hline
$i=2m-1$,  $m \not\equiv 0$ mod $p-1$  & $0{\rightarrow}\mathbb{Z}/p^{v_p(m)+1}$ \\
\hline
$i=2m-1$ and $m\equiv 0$ mod $p-1$ & $ \mathbb{Z}/p^{v_p(m)+1} \xrightarrow[]{(p,-1)} \mathbb{Z}/p^{v_p(m)+1} \oplus \mathbb{Z}/p^{v_p(m)+1} $\\
\hline
\hline
\end{tabular}
\end{center}

\begin{center}
\begin{tabular}{ | m{8cm} | m{10cm}| } 
\hline
 $\mathbf{j=2k}$ \bf{with} $ \mathbf{k\neq 0}$ & \bf{The transfer map} \\
 \hline
 $i=0$  & $\mathbb{Z}_p \xrightarrow{\cong} \mathbb{Z}_p$   \\ 
 \hline
$i=-1$ &$\ \mathbb{Z}_p  \xrightarrow[]{(p,-1)} \mathbb{Z}_p \oplus \mathbb{Z}/p^{v_p(k)}$ \\
\hline
$i=2k-1$ with $k \not\equiv 0$ mod $p-1$& $0 \rightarrow \mathbb{Z}_p $\\
\hline
$i=2k-1$ with $k \equiv 0$ mod $p-1$& $ \mathbb{Z}/p^{v_p(k)+1} \xrightarrow[]{(0,-1)} \mathbb{Z}_p \oplus \mathbb{Z}/p^{v_p(k)} $\\
\hline
$i=2m-1, m \neq 0,k$ and $m\equiv 0$ mod $p-1$ with  $v_p(k) > v_p(m)$ & $  \mathbb{Z}/p^{v_p(m)+1} \xrightarrow[]{(-1,p)} \mathbb{Z}/p^{v_p(k-m)+1} \oplus \mathbb{Z}/p^{v_p(m)+1} $\\
\hline
$i=2m-1, m \neq 0,k$ and $m\equiv 0$ mod $p-1$ with  $v_p(k) \leq v_p(m)$ & $ \mathbb{Z}/p^{v_p(m)+1} \xrightarrow[]{(p^{v_p(m)-v_p(k)}+p^{v_p(k-m)-v_p(k)})^{-1}(-1, p^{v_p(k-m)+1-v_p(k)})}\mathbb{Z}/p^{v_p(k)} \oplus \mathbb{Z}/p^{v_p(m)+v_p(k-m)+2-v_p(k)}$\\
\hline
All other $i$ & $0 \rightarrow \pi_{i,j}b(\sphere_{K(1)}) $\\
\hline
\end{tabular}
\end{center}
The displayed unit is chosen so that $\res \tr =p$ holds exactly. The odd-weight transfer is almost always zero; the sole exception is described below.
\begin{prop}
    For any $i,k \in \mathbb{Z}$, the transfer map 
    \begin{equation*}
        \tr \colon \pi_{i-1} \Map(C_{\sigma}, \sphere_{K(1)}) \rightarrow \pi_{i, 2k+1} b(\sphere_{K(1)})
    \end{equation*}
    is the zero map except when $k \neq 0$ and $i=0$. In this case, we have the transfer map
    \begin{equation*}
              \tr \colon 
        \mathbb{Z}_p \otimes \mathcal{Q} \twoheadrightarrow \mathbb{Z}/p  \xrightarrow[]{1\mapsto p^{v_p(k)}} \mathbb{Z}/p^{v_p(k)+1}.
    \end{equation*}
\end{prop}

\begin{proof}
    Dually to the restriction calculation, the transfer map factors through the $C_p$-coinvariants of $\pi_{i-1} \Map(C_{\sigma}, \sphere_{K(1)})$. As in the proof of \cref{resodd}, we can identify the $C_p$-coinvariant as 
    \begin{equation*}
        \pi_i(\sphere_{\K(1)})/p \xrightarrow[]{\cong} (\pi_{i-1} \Map(C_{\sigma}, \sphere_{\K(1)}))_{C_p}, \quad \bar{\alpha}_l \mapsto [(\bar{\alpha}_l, 0,\cdots, 0 )] 
    \end{equation*}
    Again, by degree reasons, the transfer map can be non-zero only in the following cases:
    \begin{itemize}
    \item When $k=0$ and $i=0$, $\pi_{-1} \Map(C_{\sigma}, \sphere_{\K(1)})\rightarrow \pi_{0,1} b(\sphere_{\K(1)}), \quad   \mathbb{Z}_p \otimes \mathcal{Q}\rightarrow \mathbb{Z}_p$;
    \item When $k \neq 0$ and $i=0$, $  \pi_{-1} \Map(C_{\sigma}, \sphere_{\K(1)}) \rightarrow \pi_{0,2k+1} b(\sphere_{\K(1)}), \quad \mathbb{Z}_p \otimes \mathcal{Q}  \rightarrow \mathbb{Z}/p^{v_p(k)+1} $;
    \item  When $2k+1 = mq-1$ for some integer $m$ and $i = 2k+1$, 
    $$ \pi_{2k} \Map(C_{\sigma}, \sphere_{\K(1)}) \rightarrow \pi_{2k+1,2k+1} b(\sphere_{\K(1)}), \quad \mathbb{Z}/p^{v_p(m)+1} \otimes \mathcal{Q} \rightarrow \mathbb{Z}_p .$$
\end{itemize}
For the first and the third case, the transfer maps factor through $\mathbb{Z}/p$ and map into torsion free groups, so they vanish. For the second case, we use the relation 
 $$
     \operatorname{Im} (\tr \colon \pi_{-1} \Map(C_{\sigma}, X) \rightarrow \pi_{0,2k+1}X) \cong \operatorname{Ker}(\pi_{0, 2k+1} X \rightarrow \pi_{-1,2k} F(\sphere^{\sigma} \wedge \sphere^{-\sigma}, X) )
$$
Similarly, by \cref{Ssigma}, we learn that 
\begin{equation*}
    \pi_{-1, 2k}F(\sphere^{\sigma} \wedge \sphere^{-\sigma}, b(\sphere_{\K(1)})) \cong \pi_{-1,2k} b(\sphere_{\K(1)}) \oplus (\pi_{-1}\sphere_{\K(1)})^{p-2}.
\end{equation*} 
Because $\pi_{-1} \sphere_{\K(1)} \cong \mathbb{Z}_p$ and the transfer map factors through $\mathbb{Z}/p$, we learn again that 
\begin{equation*}
    \operatorname{Im}(\tr) \cong \operatorname{Ker}(a_{\sigma} \colon \pi_{0,2k+1} b(\sphere_{\K(1)}) \rightarrow \pi_{-1,2k} b(\sphere_{\K(1)})),
\end{equation*}
From the computation table of $a_{\sigma}$-maps, the $a_{\sigma}$ table therefore forces the transfer to have the stated form.
\end{proof}

\section{Differentials in the Atiyah-Hirzebruch Spectral Sequence}
By \cref{spoketocohomo}, the spoke-graded groups are $\sphere_{\K(1)}$-cohomology groups of the $P_j$, so they can be studied by the cohomological Atiyah-Hirzebruch spectral sequences. In this section, we establish the differentials in the spectral sequence. The differentials are essentially discussed in section $1.5$ of \cite{Ravenel2003}. We provide a detailed proof here using filtered spectra and the lambda-Bockstein formalism.

We start with a quick overview of the construction of the cohomological Atiyah-Hirzebruch spectral sequence as a $\lambda$-torsion Bockstein spectral sequence. The objects are considered in the category $\operatorname{Tow}(\Sp) =\operatorname{Fun}(\mathbb{Z}^{\operatorname{op}}, \Sp)$, which is categorically the same as $\operatorname{Fil}(\Sp)$.  We use tower notation to emphasize that the spectral sequence is built from a tower rather than from an increasing filtration.

Let us consider a right-bounded tower in $\operatorname{Tow}(\Sp) =\operatorname{Fun}(\mathbb{Z}^{\operatorname{op}}, \Sp)$:
\begin{equation*}
    X_{\star} = \{\cdots \rightarrow X_{n+k} \rightarrow X_{n+k-1} \rightarrow \cdots \rightarrow X_{n} \rightarrow 0 \rightarrow \cdots \}
\end{equation*}
We denote $\operatorname{F}_s X_{\star} = \fib(X_s \rightarrow X_{s-1})$. The tower gives a spectral sequence with signature:
\begin{equation*}
    E_1^{t,s} \cong \pi_{t} (\operatorname{F}_sX) \Rightarrow \pi_t \lim X_{\star}
\end{equation*}
with differentials $d_r \colon E_r^{t,s} \rightarrow E_r^{t-1,s+r}$. Convergence requires a separate argument; below we assume strong convergence. Let $\sphere^{u, v} \in \operatorname{Fil}(\Sp)$ be a filtered spectrum starting at level $v$ as follows
\begin{equation*}
    \cdots \rightarrow 0 \rightarrow \sphere^u \xrightarrow[]{\id} \sphere^u \xrightarrow[]{\id} \cdots.
\end{equation*}
We have the filtered homotopy group defined as
\begin{equation*}
    \pi_{u,v} X_{\star} \colon = \pi_0 \operatorname{Map}(\sphere^{u,v}, X_{\star}) \cong \pi_u(X_v).
\end{equation*}
There is the canonical map $\lambda \colon \sphere^{0,-1} \rightarrow \sphere^{0,0}$ which equips filtered homotopy groups with a canonical $\mathbb{Z}[\lambda]$-structure, and we define
\begin{equation*}
    X_{\star}[\lambda^n] = \fib(X_{\star} \xrightarrow[]{\lambda^n} \Sigma^{0,n}X_{\star}).
\end{equation*}
The above spectral sequence can be rebuilt as a $\lambda$-torsion Bockstein spectral sequence by considering the following filtration:
\begin{equation*}
    0 \rightarrow X_{\star}[\lambda] \rightarrow X_{\star}[\lambda^2] \rightarrow X_{\star}[\lambda^3] \rightarrow \cdots 
\end{equation*}
Consider a $\mathbb{Z}$-module $T_{\lambda}=\frac{\mathbb{Z}[\lambda^{\pm}]}{\lambda \mathbb{Z[\lambda]}}$ and define $\Gamma_{\lambda}(X_{\star}) = \colim_{\lambda} X_{\star}[\lambda^n]$. The associated spectral sequence has signature
\begin{equation*}
    E_1^{t,s,w} \cong E_1^{t,s} \otimes T_{\lambda} \Rightarrow \pi_{t,w} (\Gamma_{\lambda}(X_{\star})) \cong \pi_{t,w}(X_{\star}),
\end{equation*}
where the last equivalence is true assuming that $\lim^1$ of the tower is $0$. The element $x\in E_1^{t,s}$ has tridegree $(t,s,s)$, the element $\lambda^{-1}$ has tridegree $(0,0,1)$ and the Bockstein differentials have the form $d_r^{\Gamma} : E_r^{t,s,w} \rightarrow E_r^{t-1, s+r, w}$. 

\begin{remark} \label{bocksteindiff}
The differentials in this spectral sequence are rigid in the sense that a differential $d_r(x) =y$ in the classical spectral sequence exists if and only if there is a Bockstein differential $d_r^{\Gamma}(\lambda^{-r}x)=y$.
By construction, an element $x \in E_1^{t,s}$ has an outgoing non-zero $d_r$-differential if and only if the image of $x$ along the natural map $\operatorname{F}_s X_{\star} \rightarrow X_s$ is $\lambda^{r-1}$-divisible but not $\lambda^r$-divisible.    
\end{remark}
 For the spectrum $P_{2n+1}$, we consider the filtered spectrum induced by its minimal skeletal filtration $P_{2n+1}^{\star}$. Applying $F(-, \sphere_{\K(1)})$, we obtain a decreasing filtration as follows:
% https://q.uiver.app/#q=WzAsMTEsWzEsMCwiRihQX3sybisxfV57Mm4raysxfSxFKSJdLFsyLDAsIkYoUF97Mm4rMX1eezJuK2t9LEUpIl0sWzMsMCwiXFxjZG90cyJdLFs0LDAsIkYoUF97Mm4rMX1eezJuKzJ9LEUpIl0sWzUsMCwiRihcXHNwaGVyZV57Mm4rMX0sRSkiXSxbNiwwLCIwIl0sWzEsMSwiRihcXHNwaGVyZV57Mm4raysxfSxFKSJdLFsyLDEsIkYoXFxzcGhlcmVeezJuK2t9LEUpIl0sWzQsMSwiRihcXHNwaGVyZV57Mm4rMn0sRSkiXSxbNSwxLCJGKFxcc3BoZXJlXnsybisxfSxFKSJdLFswLDAsIlxcY2RvdHMiXSxbMCwxXSxbMSwyXSxbMiwzXSxbMyw0XSxbNCw1XSxbNiwwXSxbNyw2LCJkXzEiLDAseyJzdHlsZSI6eyJib2R5Ijp7Im5hbWUiOiJkYXNoZWQifX19XSxbOCwzXSxbOSw0XSxbOSw4LCJkXzEiLDAseyJzdHlsZSI6eyJib2R5Ijp7Im5hbWUiOiJkYXNoZWQifX19XSxbMTAsMF0sWzcsMV0sWzEsNiwiIiwxLHsic3R5bGUiOnsiYm9keSI6eyJuYW1lIjoiZGFzaGVkIn19fV0sWzQsOCwiIiwyLHsic3R5bGUiOnsiYm9keSI6eyJuYW1lIjoiZGFzaGVkIn19fV1d
\[\begin{tikzcd}
	\cdots & {F(P_{2n+1}^{2n+k+1},\sphere_{\K(1)})} & {F(P_{2n+1}^{2n+k},\sphere_{\K(1)})} & \cdots & {F(P_{2n+1}^{2n+2},\sphere_{\K(1)})} & {F(\sphere^{2n+1},\sphere_{\K(1)})} & 0 \\
	& {F(\sphere^{2n+k+1},\sphere_{\K(1)})} & {F(\sphere^{2n+k},\sphere_{\K(1)})} && {F(\sphere^{2n+2},\sphere_{\K(1)})} & {F(\sphere^{2n+1},\sphere_{\K(1)})}
	\arrow[from=1-1, to=1-2]
	\arrow[from=1-2, to=1-3]
	\arrow[from=1-3, to=1-4]
	\arrow[dashed, from=1-3, to=2-2]
	\arrow[from=1-4, to=1-5]
	\arrow[from=1-5, to=1-6]
	\arrow[from=1-6, to=1-7]
	\arrow[dashed, from=1-6, to=2-5]
	\arrow[from=2-2, to=1-2]
	\arrow[from=2-3, to=1-3]
	\arrow["{d_1}", dashed, from=2-3, to=2-2]
	\arrow[from=2-5, to=1-5]
	\arrow[from=2-6, to=1-6]
	\arrow["{d_1}", dashed, from=2-6, to=2-5]
\end{tikzcd}\]
The associated spectral sequence has $E_1$-page
\begin{equation*}
    E_1^{t,s} \cong
    \pi_{t+s} \sphere_{\K(1)} \Rightarrow \pi_{t}F(P_{2n+1}, \sphere_{\K(1)}) \cong \pi_t L_{K(1)}(P_{2n+1}).
\end{equation*}
with differentials $d_r: E_r^{t,s} \rightarrow E_r^{t-1, s+r}$. The classes in the $E_1$-page are generated as abelian groups by $1[k] \in E_1^{-k,k}$ and $\bar{\alpha}_l[k] \in E_1^{lq-1-k,k}$. The even cells of $P_{2n+1}$ are attached to the lower cell by the map $p \in [\sphere, \sphere] \cong \mathbb{Z}$. This gives the $d_1$-differentials of the spectral sequence as $$d_1(\ba_l[2k-1])= p\ba_l[2k], \quad d_1(1[2k-1]) =p 1[2k]$$ 
and produces the $E_2$-page as
\begin{equation*}
    E_2^{t,s} \cong \mathbb{Z}/p \{\ba_{l}[2k], \alpha_l[2k-1], 1[2k] | \text{ for } k \geq n+1, l\in \mathbb{Z} \text{ except for } {\alpha}_0[2k-1] \}.
\end{equation*}
We will show the following theorem:
\begin{theorem} \label{BCPAHSS}
    The surviving cycles in this spectral sequence are
    \begin{itemize}
        \item At stem $t=0, -1$, all the classes ${\alpha}_{l}[lq-1]$, $\ba_l[lq]$ for $lq \geq 2n+2$. If $n \leq -1$, there is also the class $1[0]$.
        \item At stem $t= 2k-1$ with some integer $k \neq 0$, the lowest $v_p(k)+1$ classes are permanent cycles. To be more specific, we let $d \in \{1, \cdots, p-1\}$ be the integer such that $d \equiv -k-n \mod{p-1}$. Then the permanent classes are $ \ba_{\frac{2k+2n+2d}{q}+i}[2n+2d+iq]$ for $0 \leq i \leq v_p(k)$.
    \end{itemize}
   The differentials in this spectral sequence have the following two types:
    \begin{itemize}
        \item For the classes $1[2k]$ with $k \neq0$, they support differentials of length $r= (v_p(k)+1)q$ as
        \begin{equation*}
            d_{r}(1[2k]) = \ba_{v_p(k)+1} [2k+r].
        \end{equation*}
        \item  At stem $t=2k$ with $k \neq 0$, all the other classes except for $1[-2k]$ support differentials of length $r = (v_p(k)+1)q+1$. Namely, for a class ${\alpha}_l[lq-2k-1]$ in stem $2k$, it supports a differential
        \begin{equation*}
            d_{r} ({\alpha}_l[lq-2k-1]) = \ba_{l+v_p(k)+1}[lq-2k-1+r]
        \end{equation*}
    \end{itemize}
\end{theorem}
   \cref{fig:AHSS_chart} displays the spectral sequence for $P_{-3}$ at $p=3$:
\begin{figure}[H]
   \centering
   \scalebox{0.8}{\input{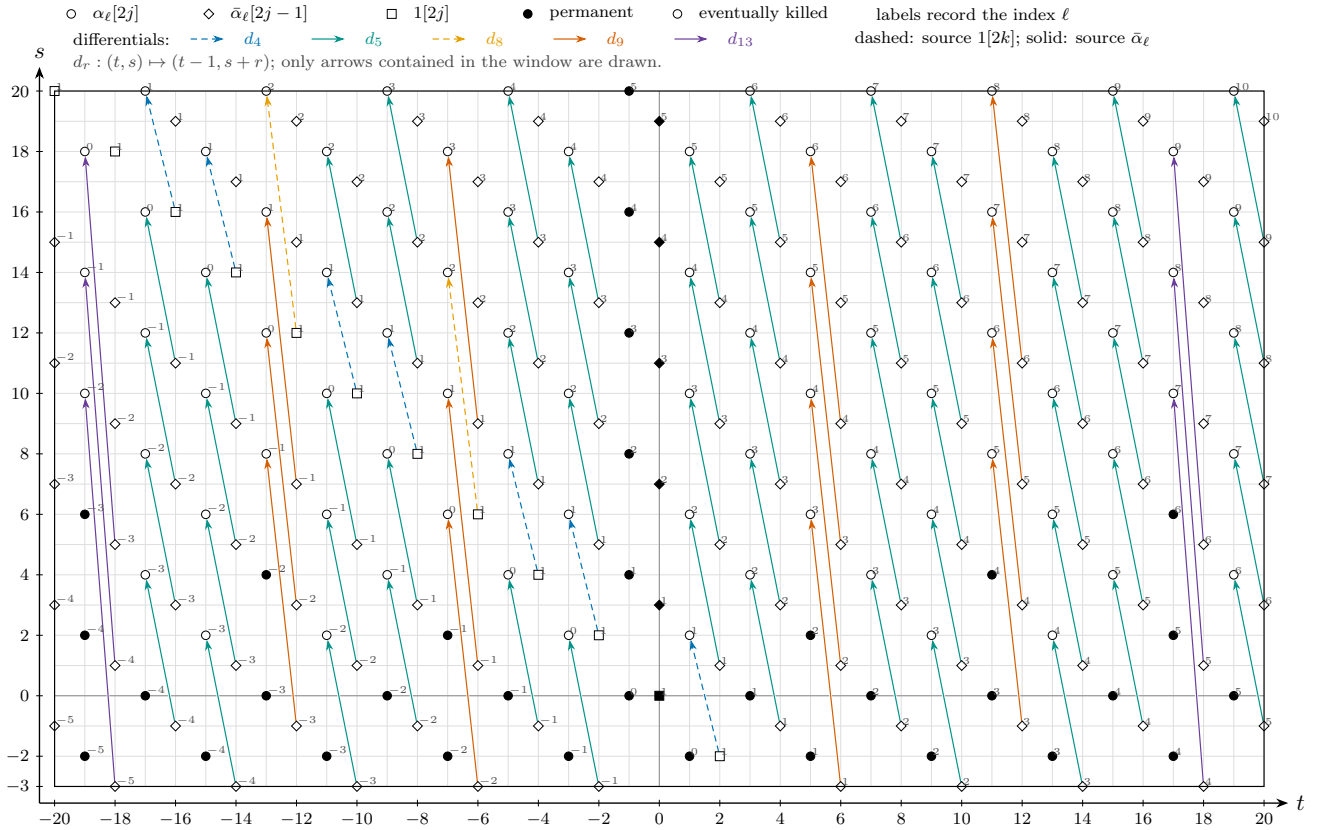}}
   \caption{Atiyah-Hirzebruch spectral sequence for $\sphere_{\K(1)}^*(P_{-3})$ at $p=3$, with $-20 \leq t \leq 20, -3 \leq s \leq 20$.}
   \label{fig:AHSS_chart}
\end{figure}

\begin{proof}
First, note that the surviving classes are forced once we prove the existence of the differentials. The first family of differentials follows directly from \cref{attachcell}. Because the primary coattaching map of cell $2k$ is an element in the image of $J$, they must be detected non-trivially by those differentials in the spectral sequence.

For the second type of differentials, we will first show the following claim:
  \begin{claim}
     For $l \in \mathbb{Z}$, set $r=v_p(lq-2n-2)+1$. If $r \neq \infty$, then the class ${\alpha}_l[2n+1] \in E_2^{lq-2n-2,2n+1}$ is $\lambda^{(r-1)q+1}$-divisible and not $\lambda^{rq+1}$-divisible.
  \end{claim}

\begin{proof}[Proof of the claim]
Since we are considering the classes in the bottom cell, the map $F_sX_{\star} \rightarrow X_{s}$ in the Atiyah-Hirzebruch filtration is identity. The $\lambda^N$-map is concretely
\begin{equation*}
    \pi_{lq-2n-2} F(P_{2n+1}^{2n+1+N}, \sphere_{\K(1)}) \rightarrow \pi_{lq-2n-2} F(\sphere^{2n+1}, \sphere_{\K(1)})
\end{equation*}
When $N=1$, this map is $\mathbb{Z}/p \rightarrow \mathbb{Z}/p^{v_p(l)+1}$, and only ${\alpha}_l$ is $\lambda$-divisible. By \cref{Xinoextension}, the $\lambda^N$-map has the following concrete description:
\begin{itemize}
    \item For $1 \leq a \leq r-1$, the map is the natural surjection:
    $$\lambda^q \colon \pi_{lq-2n-2} F(P_{2n+1}^{2n+2+aq}, \sphere_{\K(1)}) \cong \mathbb{Z}/p^{a+1} \twoheadrightarrow \pi_{lq-2n-2} F(P_{2n+1}^{2n+2+(a-1)q}, \sphere_{\K(1)}) \cong \mathbb{Z}/p^{a}.$$
    \item For $a \geq r$, the map is as follows:
     $$\lambda^q \colon \pi_{lq-2n-2} F(P_{2n+1}^{2n+2+aq}, \sphere_{\K(1)}) \cong \mathbb{Z}/p^{r} \xrightarrow[]{\times p} \pi_{lq-2n-2} F(P_{2n+1}^{2n+2+(a-1)q}, \sphere_{\K(1)}) \cong \mathbb{Z}/p^{r}.$$
\end{itemize}
The first observation implies that the $\lambda^{(r-1)q}$-map mapping to $F(P_{2n+1}^{2n+2}, \sphere_{\K(1)})$ is a surjection, so the class ${\alpha}_l[2n+1]$ must be $\lambda^{(r-1)q+1}$-divisible. Combining the two observations above, the $\lambda^{rq+1}$-map factors through the zero map as follows:
\begin{equation*}
    \pi_{lq-2n-2} F(P_{2n+1}^{2n+2+rq}) \cong \mathbb{Z}/p^{r} \xrightarrow[]{\times p} \pi_{lq-2n-2}F(P_{2n+1}^{2n+2}) \cong \mathbb{Z}/p,
\end{equation*}
so the class $\alpha_l[2n+1]$ cannot be $\lambda^{rq+1}$-divisible. 
\end{proof}
By the claim above, we see ${\alpha}_l[2n+1]$ cannot have an outgoing differential of length $(r-1)q+1$, and must have a differential before page $rq+1$. For degree reasons, it forces us to have the differential:
\begin{equation*}
    d_{rq+1} ({\alpha}_l[2n+1]) = \ba_{l+r}[2n+rq+2].
\end{equation*}
For the remaining differentials in the second family, we notice that the above proof works for any $n \in \mathbb{Z}$, and in particular $P_{2n+1+k}$. The natural map $P_{2n+1}  \rightarrow P_{2n+1+k}$ induces an injection of the cohomological Atiyah-Hirzebruch spectral sequence on the $E_2$-page, and the differentials on the classes ${\alpha}_l [2n+1+k]$ are obtained by naturality.
\end{proof}

This spectral-sequence formulation provides another way to compute the cohomology groups $\pi_*(P_{2n},\sphere_{\K(1)})$. Most elements supported on the $2n$-cell survive to the $E_{\infty}$-page for degree reasons, so most stems in the cohomological AHSS calculating $\pi_{kq-2n-1}(P_{2n}, \sphere_{\K(1)})$ have the shape
  
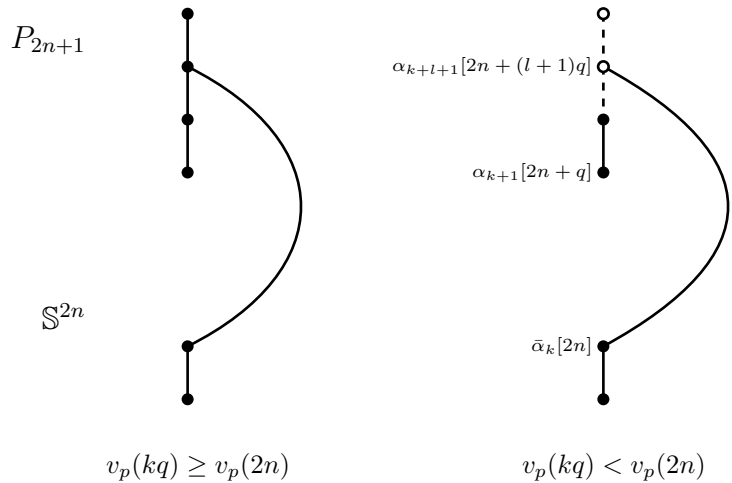
\begin{center}
   \begin{tikzpicture}[
    dot/.style ={circle, fill=black, inner sep=1.6pt},
    open/.style={circle, draw=black, fill=white,
                 inner sep=1.4pt, line width=1pt},
    every path/.style={draw=black, line width=1pt},
    x=1cm, y=1cm
]

% ---------- LEFT : S^{2n} ----------
\begin{scope}[xshift=0cm]
  \coordinate (Lt1) at (0, 5.0);  \coordinate (Lt2) at (0, 4.3);
  \coordinate (Lt3) at (0, 3.6);  \coordinate (Lt4) at (0, 2.9);
  \coordinate (Lb1) at (0, 0.6);  \coordinate (Lb2) at (0,-0.1);
  \draw (Lt1) -- (Lt2) -- (Lt3) -- (Lt4);
  \draw (Lb1) -- (Lb2);
  \draw (Lt2) .. controls +(2.0,-1.0) and +(2.0, 1.0) .. (Lb1);
  \foreach \p in {Lt1,Lt2,Lt3,Lt4,Lb1,Lb2} \node[dot] at (\p) {};
  \node[anchor=east, font=\large] at (-1.2, 4.65) {$P_{2n+1}$};
    \node[anchor=east, font=\large] at (-1.2, 1)
        {$\sphere^{2n}$};
           \node[anchor=east, font=\small] at (1.5, -1)
        {$v_p(kq) \geq v_p(2n)$};
\end{scope}

% ---------- RIGHT : P_{2n+1} ----------
\begin{scope}[xshift=5.5cm]
  \coordinate (Rt1) at (0, 5.0);  \coordinate (Rt2) at (0, 4.3);
  \coordinate (Rt3) at (0, 3.6);  \coordinate (Rt4) at (0, 2.9);
  \coordinate (Rb1) at (0, 0.6);  \coordinate (Rb2) at (0, -0.1);
  \draw (Rt3) -- (Rt4);
  \draw (Rb1) -- (Rb2); 
  \draw[dashed] (Rt1)-- (Rt2);
  \draw[dashed] (Rt2) -- (Rt3);
  \draw (Rt2) .. controls +(2.2,-1.2) and +(2.2, 1.2) .. (Rb1);
  \foreach \p in {Rt3,Rt4,Rb1,Rb2} \node[dot]  at (\p) {};
  \foreach \p in {Rt1,Rt2}         \node[open] at (\p) {};
        \node[anchor=east, font=\small] at (1.5, -1)
        {$v_p(kq) < v_p(2n)$};
         \node[anchor=east, font=\tiny] at (0, 0.6)
        {$\bar{\alpha}_k[2n]$};
            \node[anchor=east, font=\tiny] at (0, 4.3)
        {${\alpha}_{k+l+1}[2n+(l+1)q]$};
    \node[anchor=east, font=\tiny] at (0, 2.9)
        {${\alpha}_{k+1}[2n+q]$};
\end{scope}
\end{tikzpicture}
\end{center}
The Toda bracket calculation indicates a possible extension as the curved line, and depending on whether the relevant cycle survives, we can determine the extension class. This argument is made precise in the proof of \cref{P2ncalc}.

\bibliographystyle{amsalpha}
\bibliography{ref}

@article{hopkins1994constructions,
  title={Constructions of elements in {P}icard groups},
  author={Hopkins, Michael J and Mahowald, Mark and Sadofsky, Hal},
  journal={Contemporary Mathematics},
  volume={158},
  pages={89--126},
  year={1994},
  publisher={American Mathematical Society}
}

@inproceedings{harris1988stable,
  title={Stable decompositions of classifying spaces of finite abelian p-groups},
  author={Harris, John C and Kuhn, Nicholas J},
  booktitle={Mathematical Proceedings of the Cambridge Philosophical Society},
  volume={103},
  number={3},
  pages={427--449},
  year={1988},
  organization={Cambridge University Press}
}

@article{Balderrama2026C2K1LocalSphere,
  author  = {Balderrama, William},
  title   = {The {$C_2$}-Equivariant {$K(1)$}-Local Sphere},
  journal = {Mathematische Zeitschrift},
  volume  = {312},
  number  = {2},
  eid     = {52},
  year    = {2026},
  doi     = {10.1007/s00209-025-03925-1}
}

@book{bruner2006h,
  title={${H_{\infty}}$-ring spectra and their applications},
  author={Bruner, Robert R and May, J Peter and McClure, James E and Steinberger, Mark},
  volume={1176},
  year={2006},
  publisher={Springer}
}

@article{hou2025c3equivariantstablestems,
      title={${C_3}$-equivariant stable stems}, 
      author={Yueshi Hou and Shangjie Zhang},
      journal={arXiv preprint arXiv:2505.10745},
      year={2025},
      eprint={2505.10745},
      archivePrefix={arXiv},
      primaryClass={math.AT},
}

@article{angelini2025spoke,
  title={Spoke topological Hochschild homology},
  author={Angelini-Knoll, Gabriel and Zou, Foling},
  journal={arXiv preprint arXiv:2512.11338},
  year={2025}
}

@article{gonzalez1996classification,
  title={Classification of the stable homotopy types of stunted lens spaces for an odd prime},
  author={Gonzalez, Jesus},
  journal={Pacific Journal of Mathematics},
  volume={176},
  number={2},
  pages={325--343},
  year={1996},
  publisher={Mathematical Sciences Publishers}
}

@book{Ravenel2003,
  author    = {Douglas C. Ravenel},
  title     = {Complex Cobordism and Stable Homotopy Groups of Spheres},
  edition   = {2},
  series    = {AMS Chelsea Publishing},
  volume     = {347},
  publisher = {American Mathematical Society},
  address   = {Providence, RI},
  year      = {2003},
  isbn       = {9780821829677}
}

@article{HHR17,
  author  = {Hill, Michael A. and Hopkins, Michael J. and Ravenel, Douglas C.},
  title   = {The slice spectral sequence for certain
             {$RO(C_{p^n})$}-graded suspensions of
             {$H\underline{\mathbb{Z}}$}},
  journal = {Bolet{\'i}n de la Sociedad Matem{\'a}tica Mexicana},
  volume  = {23},
  number  = {1},
  pages   = {289--317},
  year    = {2017},
  doi     = {10.1007/s40590-016-0129-3},
  eprint  = {1510.06056},
  archivePrefix = {arXiv},
  primaryClass  = {math.AT}
}

@article{HSW23,
  author  = {Hahn, Jeremy and Senger, Andrew and Wilson, Dylan},
  title   = {Odd primary analogs of real orientations},
  journal = {Geometry \& Topology},
  volume  = {27},
  number  = {1},
  pages   = {87--129},
  year    = {2023},
  doi     = {10.2140/gt.2023.27.87},
  eprint  = {2009.12716},
  archivePrefix = {arXiv},
  primaryClass  = {math.AT}
}

@article{Bousfield1979Localization,
  author  = {A. K. Bousfield},
  title   = {The Localization of Spectra with Respect to Homology},
  journal = {Topology},
  volume  = {18},
  number  = {4},
  pages   = {257--281},
  year    = {1979},
  doi     = {10.1016/0040-9383(79)90018-1}
}

@article{Ravenel1984Localization,
  author  = {Douglas C. Ravenel},
  title   = {Localization with Respect to Certain Periodic Homology Theories},
  journal = {American Journal of Mathematics},
  volume  = {106},
  number  = {2},
  pages   = {351--414},
  year    = {1984},
  doi     = {10.2307/2374308}
}

@article{BehrensCarlisle2025,
  author  = {Behrens, Mark and Carlisle, Jack},
  title   = {Periodic Phenomena in Equivariant Stable Homotopy Theory},
  journal = {The Quarterly Journal of Mathematics},
  volume  = {76},
  number  = {4},
  pages   = {1033--1104},
  year    = {2025},
  doi     = {10.1093/qmath/haaf013}
}

@article{Li2026KUGLocalSphere,
  author        = {Yingxin Li},
  title         = {On the Equivariant ${KU_G}$-Local Sphere for Finite Abelian Groups},
  journal  =   {arXiv preprint arXiv:2605.30285},
  year          = {2026},
  eprint        = {2605.30285},
  archivePrefix = {arXiv},
  primaryClass  = {math.AT}
}

@article{Balderrama2024TotalPowerOperations,
  author  = {William Balderrama},
  title   = {Total Power Operations in Spectral Sequences},
  journal = {Transactions of the American Mathematical Society},
  volume  = {377},
  number  = {7},
  pages   = {4779--4823},
  year    = {2024},
  doi     = {10.1090/tran/9073}
}

@article{NikolausScholze2018,
  author  = {Thomas Nikolaus and Peter Scholze},
  title   = {On Topological Cyclic Homology},
  journal = {Acta Mathematica},
  volume  = {221},
  number  = {2},
  pages   = {203--409},
  year    = {2018},
  doi     = {10.4310/ACTA.2018.v221.n2.a1}
}

@article{guillou2024c_2,
  title={${C_2}$-Equivariant Stable Stems},
  author={Guillou, Bertrand J and Isaksen, Daniel C},
  journal={arXiv preprint arXiv:2404.14627},
  year={2024}
}

@article{carawan2023homotopy,
  title={The homotopy of the ${KU_G}$-local equivariant sphere spectrum},
  author={Carawan, Tanner N and Field, Rebecca and Guillou, Bertrand J and Mehrle, David and Stapleton, Nathaniel J},
  journal={Journal of Homotopy and Related Structures},
  volume={18},
  number={4},
  pages={543--561},
  year={2023},
  publisher={Springer}
}

@article{thompson1990,
  title={The $v_1$-periodic homotopy groups of an unstable sphere at odd primes},
  author={Thompson, Robert D},
  journal={Transactions of the American Mathematical Society},
  volume={319},
  number={2},
  pages={535--559},
  year={1990}
}

@article{BelmontGuillouIsaksen2021Comparison,
  author        = {Belmont, Eva and Guillou, Bertrand J. and Isaksen, Daniel C.},
  title         = {{$C_2$}-equivariant and {$\mathbb{R}$}-motivic stable stems {II}},
  journal       = {Proceedings of the American Mathematical Society},
  volume        = {149},
  number        = {1},
  pages         = {53--61},
  year          = {2021},
  doi           = {10.1090/proc/15167},
  eprint        = {2001.02251},
  archivePrefix = {arXiv},
  primaryClass  = {math.AT},
  url           = {https://arxiv.org/abs/2001.02251}
}

@article{BelmontIsaksen2022RealMotivicStems,
  author        = {Belmont, Eva and Isaksen, Daniel C.},
  title         = {{$\mathbb{R}$}-motivic stable stems},
  journal       = {Journal of Topology},
  volume        = {15},
  number        = {4},
  pages         = {1755--1793},
  year          = {2022},
  doi           = {10.1112/topo.12256},
  eprint        = {2001.03606},
  archivePrefix = {arXiv},
  primaryClass  = {math.AT},
  url           = {https://arxiv.org/abs/2001.03606}
}

\end{document}